\documentclass[a4paper]{article}
\usepackage{amssymb}
\usepackage{amsmath}
\usepackage[all]{xy}
\usepackage{amsthm}
\usepackage{verbatim}
\usepackage{hyperref}

\newcommand{\C}{\mathbb{C}}
\newcommand{\Q}{\mathbb{Q}}
\newcommand{\R}{\mathbb{R}}
\newcommand{\Z}{\mathbb{Z}}
\newcommand{\ideal}[1]{\langle#1\rangle}
\newcommand{\SL}{\mathrm{SL}}
\newcommand{\AGM}{\mathrm{AGM}}
\renewcommand{\H}{\mathcal{H}}
\newcommand{\F}{\mathcal{F}}

\newcommand{\RR}{\mathcal{R}}
\newcommand{\LL}{\mathcal{L}}
\renewcommand{\P}{\mathbb{P}}
\def\Magma{\textsf{MAGMA}}
\def\Sage{\textsf{Sage}}

\theoremstyle{plain}
\newtheorem{theorem}{Theorem}
\newtheorem{proposition}[theorem]{Proposition}
\newtheorem{lemma}[theorem]{Lemma}
\newtheorem{corollary}[theorem]{Corollary}
\theoremstyle{definition}
\newtheorem{example}{Example}
\newtheorem*{example*}{Example}
\newtheorem{algorithm}[theorem]{Algorithm}
\theoremstyle{remark}
\newtheorem*{remark}{Remark}

\title{The complex AGM, periods of elliptic curves over~$\C$ and complex
  elliptic logarithms} 
\author{John E. Cremona\\
Thotsaphon Thongjunthug}

\begin{document}
\maketitle
\begin{abstract}
We give an account of the complex Arithmetic-Geometric Mean (AGM), as
first studied by Gauss, together with details of its relationship with
the theory of elliptic curves over~$\C$, their period lattices and
complex parametrisation.  As an application, we present efficient
methods for computing bases for the period lattices and elliptic
logarithms of points, for arbitrary elliptic curves defined over $\C$.
Earlier authors have only treated the case of elliptic curves defined
over the real numbers; here, the multi-valued nature of the complex
AGM plays an important role.  Our method, which we have implemented in
both \Magma\ and \Sage, is illustrated with several examples using
elliptic curves defined over number fields with real and complex
embeddings.
\end{abstract}

\section{Introduction}\label{sec:intro}
Let $E$ be an elliptic curve defined over $\C$, given by a Weierstrass
equation
$$
E:\quad Y^2 = 4(X-e_1)(X-e_2)(X-e_3),
$$ where the roots $e_j\in\C$ are distinct.
As is well known, there is an isomorphism (of complex analytic Lie
groups) $\C/\Lambda\cong E(\C)$, where~$\Lambda$ is the \emph{period
  lattice} of~$E$: specifically, we take $\Lambda$ to be the lattice
of periods of the invariant differential $dX/Y$ on~$E$.  It is a
discrete rank~$2$ subgroup of~$\C$, spanned by a $\Z$-basis
$\{w_1,w_2\}$ with $w_2/w_1\notin\R$.  The isomorphism is given by the
map
$$
z\pmod\Lambda \mapsto P=(\wp_\Lambda(z), \wp'_\Lambda(z))\in E(\C)
$$ (with $0\pmod\Lambda\mapsto O\in E(\C)$, the base point at
infinity) where $\wp_\Lambda$ denotes the classical elliptic
Weierstrass function associated to the lattice~$\Lambda$.  The inverse
of this map,
\[
   P \mapsto z\pmod\Lambda,
\] 
from $E(\C)$ to $\C/\Lambda$, is called the \emph{elliptic logarithm},
and we say that any $z\in\C$ representing its class modulo~$\Lambda$
is an \emph{elliptic logarithm} of $P$.  Two natural questions are:
\begin{enumerate}
\item How can we compute a basis for the period lattice $\Lambda$ of
  $E$, given a Weierstrass equation?
\item Given a point $P=(x,y)\in E(\C)$, how can we compute its
  elliptic logarithm $z\in\C$?
\end{enumerate}

For elliptic curves over $\R$, these questions have been answered
satisfactorily and are well-known.  Algorithms for computing
$\Z$-bases for period lattices of elliptic curves defined over~$\R$,
and elliptic logarithms of real points on such curves, may be found in
the literature (see, for example, \cite[Algorithm 7.4.8]{coh} or
\cite[\S3.7]{JCbook2}).  These use the real \emph{arithmetic-geometric
  mean} (AGM), and allow one to compute both values rapidly with a
high degree of precision.  The theory behind this method is described
succinctly by Mestre in \cite{bos-mes}.  The situation for elliptic
curves over $\C$, however, is less satisfactory.

In this paper, we will give a complete method for computing period
lattices and elliptic logarithms for elliptic curves over $\C$, by
generalising the real algorithm.  To this end, we will first explain
the connection between the following three classes of objects:
\begin{itemize}
\item Complex AGM sequences, as first studied by Gauss and explored in
  depth more recently by Cox \cite{cox};
\item Chains of lattices in~$\C$;
\item Chains of $2$-isogenies between elliptic curves defined over~$\C$.
\end{itemize}
These will be defined precisely below.  This connection will allow us
to derive an explicit formula (see Theorems~\ref{thm:w1}
and~\ref{thm:w1w2w3} below), based on so-called \emph{optimal} complex
AGM values, for a $\Z$-basis of the period lattice of any elliptic
curve defined over~$\C$. We then develop our method further to give an
iterative method (Algorithm~\ref{algo:elog}) for computing elliptic
logarithms of complex points. 

Our approach to the computation of periods follows closely that of
Bost and Mestre \cite{bos-mes} in the real case.  However, in that
case there is only a single chain of $2$-isogenies which needs to be
considered, and a unique AGM sequence, while over~$\C$ we find it
convenient to consider a whole class of such sequences.  The
connection between these three types of sequence has some independent
interest.

We note that the recent paper \cite{Dupont} and thesis
\cite{Dupont-thesis} by Dupont also presents related methods for
evaluating modular functions using the complex~AGM, including explicit
complexity results (see \cite[Prop.~3.3]{Dupont-thesis}).  The results
in~\cite{Dupont} may also be used to compute complex periods, as these
are given by elliptic integrals with complex parameters.

In the next three sections of the paper we consider in turn complex
AGM sequences (as are described well in Cox \cite{cox}), then lattice
chains and finally chains of $2$-isogenies.  Then we give the first
application, to the computation of a basis for the period lattice (see
Theorem~\ref{thm:w1w2w3}).  The following section gives a new proof of
a result about the complete set of values of the (multi-valued)
complex AGM, slightly more general than the version in~\cite{cox}.
Then in Section~\ref{sec:elog} we develop the elliptic logarithm
algorithm (Algorithm~\ref{algo:elog}).  The paper ends with a set of
illustrative examples, using curves defined over number fields with
real and complex embeddings, and remarks on the algorithms'
efficiency.

Our algorithms have been implemented by the authors both in
\Sage\ (see \cite{sage}) and in \Magma\ (see \cite{magma}, code
available from the second author).

The results of this paper form part of the PhD thesis
\cite{NookThesis} of the second author.  The proofs are in some cases
different: in \cite{NookThesis} both the periods and elliptic
logarithms are expressed more traditionally, as integrals over the
Riemann surface~$E(\C)$; however the resulting iterative algorithms
are identical.  The second author acknowledges the support of the
Development and Promotion of Science and Technology Talent Project
(DPST) of the Ministry of Education, Thailand.

\section{AGM Sequences}\label{sec:agmseq}
Let $(a,b)\in\C^2$ be a pair of complex numbers satisfying
\begin{equation}\label{eqn:ab-conditions}
a\ne0, \quad b\ne0, \quad a\ne\pm b.
\end{equation}
We say that $(a,b)$ is \emph{good} if $\Re(b/a)\ge0$, or equivalently,
\begin{equation}\label{eqn:good-defn}
 |a-b| \le |a+b|;
\end{equation}
otherwise the pair is said to be \emph{bad}. Clearly, only 
one of the pairs $(a,b)$, $(a,-b)$ is good, unless $\Re(b/a)=0$ (or
equivalently, $|a-b|=|a+b|$), in which case both are good.

An \emph{arithmetic-geometric mean} (AGM) \emph{sequence} is a
sequence $((a_n, b_n))_{n=0}^\infty$, whose pairs $(a_n,b_n)\in\C^2$ satisfy
the relations
\[
2a_{n+1}= a_n+b_n,\quad
b_{n+1}^2 = a_nb_n
\]
for all $n\ge0$. It is easy to see that if any one pair $(a_n,b_n)$ in
the sequence satisfies \eqref{eqn:ab-conditions} then all do, and we
will make this restriction henceforth.

From any given starting pair $(a_0,b_0)$ there are uncountably many
AGM sequences, obtained by iterating the procedure of replacing
$(a_n,b_n)$ by the arithmetic mean $a_{n+1}=(a_n+b_n)/2$ and the
geometric mean $b_{n+1}=\sqrt{a_nb_n}$, with a choice of the square
root for $b_{n+1}$ at each step. However, we usually prefer to
consider the entire sequence as a whole.  We say that an AGM sequence
is \emph{good} if the pairs $(a_n,b_n)$ are good for all but finitely
many $n$\label{page:good-agm}. A good AGM sequence in which
$(a_n,b_n)$ are good for all $n>0$ is said to be \emph{optimal}, and
\emph{strongly optimal} if in addition $(a_0,b_0)$ is good. If an AGM
sequence is not good, then we say that it is \emph{bad}.

It is easy to check that $(a_{n+1},\pm b_{n+1})$ are both good if and
only if $a_n/b_n$ is real and negative, in which case $(a_n,b_n)$ is
certainly bad.  In an optimal sequence, this situation can only occur
for $n=0$.  In consequence, for every starting pair $(a_0,b_0)$ there
is exactly one optimal AGM sequence, unless $a_0/b_0$ is real and
negative, in which case there are two, with different signs of $b_1$,
with the property that the ratios $a_n/b_n$ in one of the sequences
are the complex conjugates of those in the other.

The following proposition is from Cox (see \cite{cox}); the proof of
parts (1) and~(2) is elementary, and we refer the reader to
\cite{cox}; part (3) appears deeper, and we will give a proof below
after relating the different AGM values to a certain set of periods of
an elliptic curve.  Note that Cox defines the notion of ``good'' more
strictly than above (when $\Re(a/b)=0$ he requires $\Im(a/b)>0$, so
that exactly one of $(a,\pm b)$ is good in every case), but in view of
the preceding remarks this does not affect the following result.

\begin{proposition}\label{prop:cox}
Given a pair $(a_0,b_0)\in\C^2$ satisfying \eqref{eqn:ab-conditions}, every AGM
sequence $((a_n,b_n))_{n=0}^\infty$ starting at $(a_0,b_0)$ satisfies
the following:
\begin{enumerate}
\item $\lim_{n\to\infty}a_n$ and $\lim_{n\to\infty}b_n$ exist and are equal;
\item The common limit, say $M$, is non-zero if and only if the
sequence is good;
\item $|M|$ attains its maximum (among all AGM-sequences starting at
  $(a_0,b_0)$) if and only if the sequence is optimal.
\end{enumerate}
\end{proposition}

For an AGM sequence $((a_n,b_n))_{n=0}^\infty$ starting at
$(a_0,b_0)$, we will denote the common limit
$\lim_{n\to\infty}a_n=\lim_{n\to\infty}b_n$ by $M_S(a_0,b_0)$, where
$S\subseteq\Z_{>0}$ is the set of all indices $n$ for which the pair
$(a_n,b_n)$ is bad.  For example, $M_{\emptyset}(a_0,b_0)$ denotes the
common limit for the optimal AGM sequence.  To avoid ambiguities when
$a_0/b_0$ is negative real, we may agree to choose $b_1$ so that
$\Im(a_1/b_1)>0$ in that case, though this choice will not affect our
results below.  Note that the AGM sequence is good if and only if $S$
is a finite set. To ease notation, we shall write
$M_{\emptyset}(a_0,b_0)$ simply as $M(a_0,b_0)$.

\section{Lattice Chains}\label{sec:lattice-chain}
In this paper, a \emph{lattice} will always be a free $\Z$-module of
rank $2$, embedded as a discrete subgroup of $\C$. Elements of
lattices will often be called periods, since in our application the
lattices will arise as period lattices of elliptic curves defined over
$\C$.  

The following definition, as well as Lemma~\ref{lem:primitive}, only
depend on the algebraic structure of lattices.
We define a \emph{chain of lattices (of index $2$)} to be a sequence
of lattices $(\Lambda_n)_{n=0}^\infty$ which satisfies the following
conditions:
\begin{enumerate}
\item $\Lambda_n \supset \Lambda_{n+1}$ for all $n\ge0$;
\item $[\Lambda_n:\Lambda_{n+1}]=2$ for all $n\ge0$;
\item $\Lambda_0/\Lambda_n$ is cyclic for all $n\ge1$; equivalently,
 $\Lambda_{n+1}\ne 2\Lambda_{n-1}$ for all $n\ge1$.
\end{enumerate}
Thus for each $n\ge1$ we have
\begin{equation}\label{eqn:next-lambda}
\Lambda_{n+1}=\ideal{w}+2\Lambda_n
\end{equation}
for some $w\in\Lambda_n\setminus2\Lambda_{n-1}$.  Given an initial
lattice $\Lambda_0$, there are three possibilities for
$\Lambda_1$. When $n\ge1$, one of the three sublattices of index~$2$
is excluded, since it is contained in $2\Lambda_{n-1}$ (which would
contradict the last condition in the definition), and so there are
only two possible choices for $\Lambda_{n+1}$.  The number of such
chains starting with $\Lambda_0$ is uncountable; we will distinguish a
countable subset of these as follows. Let
$$
\Lambda_\infty = \bigcap_{n=0}^\infty \Lambda_n.
$$ Then $\Lambda_\infty$ is free of rank at most $1$; the rank cannot
be $2$, since for all $n$,
$$
[\Lambda_0:\Lambda_\infty]\ge[\Lambda_0:\Lambda_n]=2^n,
$$ so $[\Lambda_0:\Lambda_\infty]$ is infinite.  We say that the chain
is \emph{good} if $\Lambda_\infty$ has rank $1$; in this case a
generator for $\Lambda_\infty$ will be called a \emph{limiting period}
of the chain.  We will first show that the limiting period is
\emph{primitive}, in the sense that it is not in $m\Lambda_0$ for any
$m\ge2$.

\begin{lemma}\label{lem:primitive}
Let $(\Lambda_n)_{n=0}^\infty$ be a good chain with $\Lambda_\infty =
\ideal{w_\infty}$.  Then 
\begin{enumerate}
\item $w_\infty$ is primitive; equivalently,
  $\Lambda_0/\Lambda_\infty$ is free of rank $1$;
\item $\Lambda_n=\ideal{w_\infty}+2^n\Lambda_0$ for all $n\ge0$.
\end{enumerate}
\end{lemma}

\begin{proof}
Suppose that $w_\infty=mw$ for some $m\ge1$ and $w\in\Lambda_0$.  If
$m$ is odd, then since $\Lambda_0/\Lambda_n$ has order $2^n$ which is
prime to $m$, we see that
$$
mw\in\Lambda_n\implies w\in\Lambda_n
$$
for all $n$, so that $w\in\Lambda_\infty$.  Hence (by definition
of $w_\infty$), $m=1$.

Next suppose that $w_\infty=2w$ for some $w\in\Lambda_0$.  By
definition of $w_\infty$, we then have $w\notin\Lambda_\infty$, and
hence there exists $n>0$ such that $w\notin\Lambda_n$. This implies
that $w_\infty\in\Lambda_{n}\setminus2\Lambda_{n}$. But since
$w_\infty\in\Lambda_{n+1}$, we have
$$
\Lambda_{n+1} = \ideal{w_\infty}+2\Lambda_{n} =
\ideal{2w}+2\Lambda_{n+1}\subseteq2\Lambda_0,
$$
which contradicts the definition of a chain. This proves the
first statement.

The second statement follows from the fact
that $\Lambda_n/2^n\Lambda_0$ is cyclic of order $2^n$, and is generated
by $w_\infty$ modulo $2^n\Lambda_0$, since $w_\infty$ is primitive.
\end{proof}

So far, our notion of a good chain has been defined as a property of
the chain as a whole, and only used the abstract structure of lattices
as free $\Z$-modules. Using the next definition, we will see that this
property can also be seen in terms of the individual steps
$\Lambda_{n}\supset\Lambda_{n+1}$, when the lattices are embedded in
$\C$.  In view of \eqref{eqn:next-lambda}, the choice of
$\Lambda_{n+1}$ is determined by the class of $w$ modulo $2\Lambda_n$.

For $n\ge1$, we say that $\Lambda_{n+1}\subset\Lambda_n$ is the
\emph{right choice} of sublattice of $\Lambda_n$ if
$\Lambda_{n+1}=\ideal{w}+2\Lambda_n$
where $w$ is a \emph{minimal} element in $\Lambda_n\setminus2\Lambda_{n-1}$
(with respect to the usual complex absolute value).

\begin{lemma}\label{lem:minimal}
Let $(\Lambda_n)_{n=0}^\infty$ be a good chain with $\Lambda_\infty =
\ideal{w_\infty}$.  Then $w_\infty$ is minimal in $\Lambda_n$ for all
but finitely many $n\ge0$.
\end{lemma}

\begin{proof}
Since $\Lambda_0$ is discrete, the number of periods $w\in\Lambda_0$
with $0<|w|<|w_\infty|$ is finite.  Each of these periods lies in
only finitely many~$\Lambda_n$ by minimality of~$w_\infty$
in~$\Lambda_\infty$, so there exists $n_0$ such that~$w_\infty$ is
minimal in $\Lambda_{n_0}$ and hence also in  $\Lambda_{n}$ for
all~$n\ge n_0$.
\end{proof}

The following proposition yields an alternative notion of a good
chain.  For now we remark that this is analogous to the definition of
a good AGM sequence in the previous section; more of its analogues
will be seen in later sections.
\begin{proposition}\label{prop:finitely-many-chain}
A chain of lattices $(\Lambda_n)_{n=0}^\infty$ is good if and only if
$\Lambda_{n+1}\subset\Lambda_n$ is the right choice for all but
finitely many $n\ge1$.
\end{proposition}

\begin{proof}
Let $(\Lambda_n)_{n=0}^\infty$ be a good chain with
$\Lambda_\infty=\ideal{w_\infty}$. Then by Lemma \ref{lem:minimal},
there exists an integer $n_0$ such that $w_\infty$ is minimal in $\Lambda_n$
for all $n\ge n_0$. Since $\Lambda_{n+1}=\ideal{w_\infty}+2\Lambda_{n}$ for all
$n$, then by definition, $\Lambda_{n+1}\subset\Lambda_n$ is the right choice
for all $n\ge n_0$.

Conversely, suppose that $\Lambda_{n+1}\subset\Lambda_n$ is the right
choice for all $n\ge n_0$ (where $n_0\ge1$).  Without loss of
generality, we may suppose that $n_0=1$.  Let $w_1\in\Lambda_1$ be a
minimal element. Then $w_1$ is certainly primitive (as an element of
$\Lambda_1$, though not necessarily in $\Lambda_0$).  We claim that
$w_1\in\Lambda_n$ for all $n\ge1$, so that the chain is good with
limiting period $w_1$.

To prove the claim, suppose that $w_1\in\Lambda_j$ for all $j\le n$.
Then $\Lambda_n=\ideal{w_1}+2^{n-1}\Lambda_1$, since the latter is
contained in the former and both have index $2^{n-1}$ in $\Lambda_1$.
Hence $\Lambda_n=\ideal{w_1,2^{n-1}w_2}$, where $w_2\in\Lambda_1$ is
such that $\Lambda_1=\ideal{w_1,w_2}$.  The right sublattice
of $\Lambda_{n+1}$ is clearly $\ideal{w_1}+\Lambda_n$, by minimality
of $w_1$ (which is a candidate since
$w_1\in\Lambda_n\setminus2\Lambda_{n-1}$); in particular,
$w_1\in\Lambda_{n+1}$, as required.
\end{proof}

\subsection{Optimal Chains and Rectangular Lattices}
Let us define a lattice chain to be \emph{optimal} if
$\Lambda_{n+1}\subset\Lambda_n$ is the right choice for \emph{all}
$n\ge1$.  We will see that there is in general just one optimal chain
for each of the three choices of $\Lambda_1\subset\Lambda_0$.
In order to make the statement more precise, however, some preparation
is necessary. 

We say that a lattice $\Lambda\subset\C$ is \emph{rectangular} if it
has an ``orthogonal'' $\Z$-basis $\{w_1, w_2\}$, meaning  one which
satisfies $\Re(w_2/w_1)=0$.  For example, the period lattice of an
elliptic curve defined over $\R$ with positive discriminant is
rectangular, where an orthogonal basis is given by the least real
period and the least imaginary period.  In general, rectangular
lattices are homothetic to the period lattices of this family of
elliptic curves.

If $\{w_1,w_2\}$ is any $\Z$-basis for a lattice $\Lambda$, the three
non-trivial cosets of $2\Lambda$ in $\Lambda$ are $C_j=w_j+2\Lambda$
for $j=1, 2, 3$, where $w_3=w_1+w_2$.  By a \emph{minimal coset
  representative} in~$\Lambda$ we mean a minimal element of one of
these cosets.  Minimal coset representatives are always primitive; for
they are certainly not in $2\Lambda$, and if $w=mw'$ with $m\ge3$ odd,
then $|w'|<|w|$ while $w'$ is in the same coset as $w$.

\begin{lemma}\label{lem:ortho-basis}
In each coset $C_j$ the minimal coset representative is unique up to
sign, except in the case of a rectangular lattice with orthogonal
basis $\{w_1,w_2\}$ where the coset $C_3$ has four minimal vectors,
$\pm(w_1\pm w_2)$. 
\end{lemma}
\begin{proof}
For a rectangular lattice with orthogonal
basis $\{w_1, w_2\}$, it is easy to see that the minimal coset representatives
are as stated. Conversely, suppose that the lattice $\Lambda$ has a coset $C$
with at least two pairs of minimal elements, $\pm w$ and $\pm w'$.
Then $w_1,w_2 = (w\pm w')/2\in\Lambda$ are easily seen to be
orthogonal. 

If $w_1\equiv0\pmod{2\Lambda}$, then $w_2\equiv w\pmod{2\Lambda}$. But
then $|w_2|<|w_1+w_2|=|w|$, contradicting minimality of $w$ in its
coset. Hence $w_1\not\equiv0\pmod{2\Lambda}$. Similarly,
$w_2\not\equiv0\pmod{2\Lambda}$.  Moreover, $w_1\not\equiv
w_2\pmod{2\Lambda}$ since
$w=w_1+w_2\equiv w_1-w_2\not\equiv0\pmod{2\Lambda}$. Therefore,
$w_1,w_2,w$ do represent the three non-trivial cosets modulo
$2\Lambda$. Now if $\{w_1,w_2\}$ was not a $\Z$-basis, there would
exist a non-zero period $w_0=\alpha w_1+\beta w_2$ with $0\le
\alpha,\beta <1$.  But then one of $w_0,w_0-w_1,w_0-w_2,w_0-w$ is in
the same coset as $w$, and all are smaller, contradiction.
\end{proof}

Our algorithm for computing periods of elliptic curves will in fact
compute minimal coset representatives.  Although these are
individually primitive, to ensure that we thereby obtain a $\Z$-basis
for the lattice, the following lemma is required.

\begin{lemma}\label{lem:z-basis}
For $j=1,2,3$, let $w_j$ be minimal coset representatives for a
non-rectangular lattice $\Lambda\subset\C$; that is, minimal elements
of the three non-trivial cosets of $2\Lambda$ in $\Lambda$.  Then any
two of the $w_j$ form a $\Z$-basis for $\Lambda$, and $w_3=\pm(w_1\pm
w_2)$.
\end{lemma}
\begin{proof}
We may assume that $|w_1|\le|w_2|\le|w_3|$. Then $w_1$ is minimal in
$\Lambda$ and $w_2$ is minimal in $\Lambda\setminus\ideal{w_1}$.
Hence (replacing $w_2$ by $-w_2$ if necessary) $\tau=w_2/w_1$ is in
the standard fundamental region for $\SL_2(\Z)$ acting on the upper
half-plane, $\{w_1,w_2\}$ is a $\Z$-basis, and $w_3=w_2\pm w_1$; the
sign depends on that of $\Re(\tau)$.
\end{proof}

The following proposition shows that the limiting period of an optimal chain
is closely related to minimal coset representatives.

\begin{proposition}\label{prop:min-coset-rep}
A good chain of lattices $(\Lambda_n)_{n=0}^\infty$ with
$\Lambda_\infty=\ideal{w_\infty}$ is optimal if and
only if $w_\infty$ is a minimal coset representative of $2\Lambda_0$
in $\Lambda_0$.
\end{proposition}

\begin{proof}
Suppose that $w_\infty$ is a minimal coset representative.  Then it
is clear that
$\Lambda_{n+1}=\ideal{w_\infty}+2\Lambda_n\subset\Lambda_n$ is the
right sublattice for all $n\ge1$, since $w_\infty$ is certainly
minimal in $\Lambda_n\setminus2\Lambda_{n-1}$.

Conversely, suppose that the sequence is optimal. Let $w$ be a minimal
element of $\Lambda_1\setminus2\Lambda_0$, so that $w$ is a minimal
coset representative for the unique non-trivial coset of $2\Lambda_0$
which is contained in $\Lambda_1$.  Note that $w$ is unique up to
sign, unless $\Lambda_0$ is rectangular in which case (for one of the
cosets) there will be two possibilities for $w$ up to sign. By
optimality, the sublattice $\Lambda_2\subset\Lambda_1$ is the right
choice.  In particular, if $\Lambda_0$ is not rectangular, then we
must therefore have $\Lambda_2=\ideal{w}+2\Lambda_1$.  This, however,
may not hold in the rectangular case, but it will hold if we replace
$w$ by the other choice of minimal coset representative.

Now we claim that $\Lambda_n=\ideal{w}+2\Lambda_{n-1}$ for
all $n\ge2$. We already know this for $n=2$.  If the claim is true
for $n$, then certainly $w\in\Lambda_n\setminus2\Lambda_{n-1}$
(since $w\notin2\Lambda_0$), so the (unique) good choice of sublattice
of $\Lambda_n$ is $\ideal{w}+2\Lambda_n$. By optimality, this
is $\Lambda_{n+1}$, and so the claim holds for $n+1$. Thus
$w\in\bigcap_{n=0}^\infty \Lambda_n=\ideal{w_\infty}$, and indeed,
$w=\pm w_\infty$, since $w$ is primitive.
\end{proof}

Combining Lemma \ref{lem:ortho-basis} with Proposition
\ref{prop:min-coset-rep}, we have the following conclusion.
\begin{corollary}\label{cor:lattice-conclude}
Every non-rectangular lattice $\Lambda$ has precisely three optimal
sublattice chains, whose limiting periods are the minimal coset
representatives in each of the three non-zero cosets of $2\Lambda$
in $\Lambda$. Every rectangular lattice $\Lambda$ has precisely four
optimal sublattice chains.
\end{corollary}

\section{Short lattice chains and level 4 structures}

In this section we establish a link between AGM sequences and lattice
chains.  The first step is to associate a pair of nonzero complex
number~$(a,b)$ (with $a\not=\pm b$) to each ``short'' lattice chain
$\Lambda_0\supset\Lambda_1\supset\Lambda_2$ in such a way that $(a,b)$
is good in the sense of Section~\ref{sec:agmseq} if and only if
$\Lambda_2$ is the right choice of sublattice of~$\Lambda_1$, in the
sense of Section~\ref{sec:lattice-chain}.

We establish bijections between the following sets:
\begin{enumerate}
\item ``short'' lattice chains $\Lambda_0\supset\Lambda_2$ with
  $\Lambda_0/\Lambda_2$ cyclic of order~$4$;
\item triples $(E,\omega, H)$ where $E$ is an elliptic curve defined
  over~$\C$, $\omega$ a holomorphic differential on~$E$, and $H\subset
  E(\C)$ a cyclic subgroup of order~$4$;
\item unordered pairs of nonzero complex numbers ${a,b}$ with $a^2\ne
  b^2$, where the pairs ${a,b}$ and ${-a,-b}$ are identified.
\end{enumerate}
For each short lattice chain $\Lambda_0\supset\Lambda_2$, if we set
$\Lambda_1=\Lambda_2+2\Lambda_0$ then
$(\Lambda_0,\Lambda_1,\Lambda_2)$ satisfy the conditions for the first
three terms in a lattice sequence as defined earlier.  Hence we will
usually think of a short lattice chain as a triple
$\Lambda_0\supset\Lambda_1\supset\Lambda_2$, even though $\Lambda_1$
is uniquely determined by the other two.

To each short lattice chain we associate the elliptic curve
$E=\C/\Lambda_0$ with differential~$\omega=dz$ and subgroup
$H=(1/4)\Lambda_2/\Lambda_0$.  Conversely, to a triple $(E,\omega,H)$
we associate the chain $\Lambda_0\supset\Lambda_2$ where $\Lambda_0$
is the lattice of periods of~$\omega$ (so that
$E(\C)\cong\C/\Lambda_0$), and $\Lambda_2$ is the sublattice such that
$H\cong(1/4)\Lambda_2/\Lambda_0$ under this isomorphism.

Each triple $(E,\omega,H)$ has a model of the form  
\[
   E_{\{a,b\}}:\quad Y^2 = 4X(X+a^2)(X+b^2),
\]
for some unordered pair $a,b\in\C^*$ such that $a^2\not= b^2$, where
$\omega=dX/Y$, and $H$
is the subgroup generated by the point
\[
   P_{\{a,b\}} = (ab,2ab(a+b)).
\]
The four points $P_{\{\pm a,\pm b\}}$ are the solutions to
$2P=T=(0,0)\in E_{\{a,b\}}(\C)[2]$.  Interchanging $\{a,b\}$
and~$\{-a,-b\}$ does not affect the curve and interchanges
$P_{\{a,b\}}$ and $P_{\{-a,-b\}} =-P_{\{a,b\}}$ so does not
change~$H$.  On the other hand, changing the sign of just one of~$a,b$
changes~$H$ to the other cyclic subgroup of order~$4$ containing~$T$.
Hence the pair $\{a,b\}$ has the properties stated and is well-defined
up to changing the signs of both~$a$ and~$b$,

Conversely, given $\{a,b\}$ with $ab\not=0$ and $a\not=\pm b$, we
recover the triple $(E_{\{a,b\}},\omega,H)$, which is unchanged by
either interchanging $a$ and~$b$ or negating both.

If we only consider elliptic curves up to isomorphism, we may ignore
the differential~$\omega$, scale the equations arbitrarily, and
consider lattices only up to homothety.  Now we can identify pairs
$\{a,b\}$ and $\{u a,u b\}$ for all~$u\in\C^*$.  The equation
for~$E_{\{a,b\}}$ can be scaled so that $ab=1$, giving the homogeneous
form
\[
   E_f:\quad Y^2 = 4X(X^2+fX+1),
\] 
where 
\[
   f=\frac{a^2+b^2}{ab} = \frac{a}{b}+\frac{b}{a} \not= \pm2.
\]  
In this model, the points~$\pm P_{\{a,b\}}$ generating the
distinguished subgroup~~$H$ now have coordinates $(1,\pm2\sqrt{2+f})$.
Thus the pair $E,H$ uniquely determines a complex
number~$f\in\C\setminus\{\pm2\}$.  We call this~$f$ the \emph{modular
  parameter} for the level~$4$ structure, since (as we will see below)
it is in fact the value of a modular function for the congruence
subgroup~$\Gamma_0(4)$.

\begin{proposition}
The above constructions give a bijection between these sets:
\begin{enumerate}
\item ``short'' lattice chains $\Lambda_0\supset\Lambda_2$ up to
  homothety;
\item pairs $(E, H)$ where $E$ is an elliptic curve defined over~$\C$
  with $H\subset E(\C)$ a distinguished cyclic subgroup of order~$4$,
  up to isomorphism (where isomorphisms preserve the distinguished
  subgroups);
\item complex numbers~$f\in\C\setminus\{\pm2\}$.
\item points in the open modular
  curve~$Y_0(4)=\Gamma_0(4)\backslash\H$,  where $\H$ denotes the upper
  half-plane.
\end{enumerate}
\end{proposition}

\begin{remark}
It would appear that considering pairs~$(E,P)$, where $P$ is a point
of exact order~$4$ in~$E(\C)$, would give a refinement to the
level~$4$ structure, corresponding to points on the modular
curve~$Y_1(4)=\Gamma_1(4)\backslash\H$, since
$[\Gamma_0(4):\Gamma_1(4)]=2$.  However, this is an illusion: since
every~$E$ has an automorphism $[-1]$ which takes~$P$ to~$-P$, the set
of pairs~$(E,H)$ (up to isomorphism) may be identified with the space
of pairs~$(E,P)$ (also up to isomorphism).  Similarly, since
$\Gamma_0(4)=(\pm I)\Gamma_1(4)$,  we may identify $Y_1(4)$
and~$Y_0(4)$.
\end{remark}
\begin{proof}
Bijections between all sets except the last have already been
established.

Up to homothety, the lattice~$\Lambda_0$ is determined by $\tau\in\H$
modulo the action of the modular group~$\Gamma=\SL(2,\Z)$: for any
oriented basis $w_1,w_2$ of~$\Lambda_0$ (where ``oriented'' means
$w_2/w_1\in\H$) we associate~$\tau=w_2/w_1\in\H$.  All oriented bases
$w_1',w_2'$ of~$\Lambda_0$ have the form
\[
\begin{aligned}
w_2' &= aw_2+bw_1\\
w_1' &= cw_2+dw_1\\
\end{aligned}
\]
with\footnote{The reason for ordering bases this way is to maintain
  consistency with other sections.}
$\gamma=\begin{pmatrix}a&b\\c&d \end{pmatrix}\in\Gamma$, and
$\tau'=w_2'/w_1'=(a\tau+b)/(c\tau+d)$.  To allow for the additional
level~$4$ structure, we restrict to oriented bases $w_1,w_2$ such that
$\Lambda_2=\ideal{w_1}+4\Lambda_0$, so that $w_1',w_2'$ is only
admissible if $w_1\equiv\pm w_1'\pmod{4\Lambda_0}$, or equivalently
$(c,d)\equiv(0,\pm1)\pmod{4}$.  This uniquely determines the
$\Gamma_0(4)$-orbit of~$\tau$ and not just its~$\Gamma$-orbit.

Let $Y_0(4)=\Gamma_0(4)\backslash\H$ denote the open modular curve
associated to~$\Gamma_0(4)$, and $X_0(4)$ its completion, obtained by
including the three cusps (represented by $\infty$, $0$
and~$1/2\in\P^1(\Q)$), which has genus~$0$.  Hence the function field
of~$X_0(4)$ is generated by a single function.  Since the
$j$-invariant of $E_f$ is $256{(f^2-3)^3}/{(f^2-4)}$, we see that
$f=f(\tau)$ is a suitable function.  This establishes the claim
concerning~$f$, and completes the proof of the proposition.
\end{proof}

\begin{remark}
There is an involution on each of these sets, which preserves the
level~$2$ structure but interchanges the two possible associated
level~$4$ structures.  In each of the sets this takes the following
forms: replace $\Lambda_2$ by the other sublattice~$\Lambda_2'$ of
index~$2$ in~$\Lambda_1$ such that~$\Lambda_0/\Lambda_2'$ is cyclic;
replace $H$ by the other subgroup~$H'$ which is cyclic of order~$4$
and contains~$T=(0,0)$; change the sign of one of~$a,b$; or change $f$
to~$-f$.  This involution comes from the nontrivial automorphism of
the cover $X_0(4)\rightarrow X_0(2)$ of degree~$2$; the function field
of~$X_0(2)$ is generated by~$f^2$.
\end{remark}
\begin{remark}
Since $\begin{pmatrix}2&0\\0&1\end{pmatrix}
  \Gamma_0(4) \begin{pmatrix}2&0\\0&1\end{pmatrix}^{-1} = \Gamma(2)$,
    the function field of $X_0(4)$ may also be generated by
    $\lambda(2\tau)$ where $\lambda(\tau)$ is the classical Legendre
    elliptic function which generates the function field of~$X(2)$.  A
    calculation shows that
    $f(\tau)=2(1+\lambda(2\tau))/(1-\lambda(2\tau))$.  One
    interpretation of this is that instead of parametrizing short
    lattice chains by the parameter~$\tau\in Y_0(4)$ corresponding
    to~$\Lambda_0$ with $\Gamma_0(4)$-structure, we could instead have
    used the parameter $2\tau\in Y(2)$ to parametrize the middle
    lattice~$\Lambda_1$ with full level~$2$-structure given by the
    sublattices $\frac{1}{2}\Lambda_0$ and~$\Lambda_2$.
\end{remark}

We now state the main result of this section.
\begin{theorem} \label{prop:good=right}
Let $\Lambda_0\supset\Lambda_1\supset\Lambda_2$ be a short lattice
chain corresponding to the unordered pair~$\{a,b\}$ and modular
parameter~$f$.  Then the following are equivalent:
\begin{enumerate}
\item $\Lambda_2$ is the right choice of sublattice of~$\Lambda_1$;
\item the pair $(a,b)$ is good;
\item $\Re(f)\ge0$.
\end{enumerate}
\end{theorem}

\begin{proof}
Equivalence of the second and third conditions is immediate from
$f=a/b+b/a$ since $(a,b)$ is good if and only if $\Re(a/b)\ge0$.  (In
terms of the Legendre function, the equivalent condition is that
$|\lambda(2\tau)|\le 1$.)  

For equivalence of the first condition, we need to work harder.
Recall that $\Lambda_2$ is the right choice of sublattice if
$\Lambda_2 = \ideal{w_1}+2\Lambda_1 = \ideal{w_1}+4\Lambda_0$, where
  $w_1$ is the minimal period in its coset modulo~$2\Lambda_1$.  We
  now characterize this condition in terms of the imaginary part
  of~$\tau\in\H$.
\begin{lemma}
Let $w_1,w_2$ be any oriented basis for~$\Lambda_0$ such that
$\Lambda_2 = \ideal{w_1}+4\Lambda_0$, and let $\tau=w_2/w_1$.  The
following are equivalent:
\begin{enumerate}
\item $\Im\tau$ is maximal, over all~$\tau$ in its
  $\Gamma_0(4)$-orbit;
\item $|c\tau+d|\ge1$ for all coprime~$c,d\in\Z$ such that
  $(c,d)\equiv(0,\pm1)\pmod{4}$; 
\item $|w_1|$ is minimal, over all primitive periods of~$\Lambda_0$
  such that $\Lambda_2 = \ideal{w_1}+4\Lambda_0$; 
\item $|\tau+d/4|\ge1/4$ for all odd~$d\in\Z$.
\end{enumerate}
\end{lemma}
\begin{proof}
Equivalence of the first two statements follows from
$\Im(\gamma\tau)=\Im\tau/|c\tau+d|^2$ for
$\gamma=\begin{pmatrix}a&b\\c&d \end{pmatrix}$.  Since
$|c\tau+d|\ge1\iff|cw_2+dw_1|\ge|w_1|$, the third statement is also
equivalent to these.  For the last statement, consider the geometry of
the upper half-plane: the region given by these conditions are the
same: (4) states that $\tau$ lies on or above all the semicircles
centred on rationals with denominator~$4$ and with radius~$1/4$, while
(2) says that $\tau$ lies above all semicircles centred at
rationals~$-d/c$ with radius~$1/c$ for which
$(c,d)\equiv(0,\pm1)\pmod{4}$;  as the semicircles for $c>4$ lie
strictly under those for $c=4$, this is no stronger.
\end{proof}

We denote by~$\F(4)$ the set of~$\tau$ which satisfy these conditions;
that is, those for which $\Im\tau$ is maximal in
a~$\Gamma_0(4)$-orbit.  The subset of~$\F(4)$ consisting of $\tau$
such that $0\le\Re\tau\le1$ is a (closed) fundamental region
for~$\Gamma_0(4)$.  Since $\Gamma_0(4)$ has index~$2$
in~$\Gamma_0(2)$, the region $\F(4)$ decomposes into two components,
which we denote~$\F^{\pm}(4)$: the first is $\F^+(4)=\F(2)$,
consisting of all $\tau$ lying on or above all semicircles of radius
$1/2$ centred at rationals with denominator~$2$; a closed fundamental
region for~$\Gamma_0(2)$ is the subset of~$\F(2)$ consisting of $\tau$
such that $0\le\Re\tau\le1$.  Secondly,
$\F^-(4)=\overline{\F(4)\setminus\F(2)}$.  The boundary between these
is $\F^+(4)\cap\F^-(4)$, which consists of the union of the
semicircles $|\tau+d/2|=1/2$ for all odd~$d\in\Z$.

The lemma above shows that $\Lambda_2$ is the right choice if and only
if the $\tau$ for which $\Im\tau$ is maximal over all~$\tau$ in its
$\Gamma_0(4)$-orbit is also maximal in the larger $\Gamma_0(2)$-orbit;
in other words, given that $\tau\in\F(4)$, we in fact have
$\tau\in\F^+(4)$.

The following lemma then completes the proof of the theorem.
\end{proof}
\begin{lemma}
Let $\tau\in\F(4)$.  Then
\[
    \Re f(\tau) \ge 0 \iff \tau\in\F^+(4).
\]
\end{lemma}
\begin{proof}
This will follow by continuity from the following facts, for
$\tau\in\F(4)$ with $0\le\Re\tau\le1$:
\begin{enumerate}
\item $\Re f(\tau)=0$ if and only if~$\tau$ lies on the semicircle
  $|2\tau-1| = 1$ which separates the interiors of~$\F^{\pm}(4)$.
\item $\Re f(\tau)>0$ for at least one $\tau\in\F^+(4)$.
\end{enumerate}
For the first fact implies that $\Re f(\tau)$ has constant nonzero sign
on the two interiors.  Since $f(\gamma\tau)=-f(\tau)$ for
$\gamma\in\Gamma_0(2)\setminus\Gamma_0(4)$, the signs are different in
the two interiors; and the second fact establishes that $\Re f(\tau)$
is positive in the interior of~$\F^+(4)$.

Let $f=f(\tau)$.  Let the roots of $X(X^2+fX+1)$ be~$e_1=0, e_2,e_3$.
Since $e_2e_3=1$, we have $\Re f=0$ if and only if $\Re e_2=\Re e_3=0$
with the imaginary parts $\Im e_2$, $\Im e_3$ of opposite sign; so
$e_1,e_2,e_3$ collinear, with $0$ in between the other two roots.
Conversely, if the $e_j$ are colinear with $0$ in the middle then
(since $e_2e_3=1$) it follows that~$\Re f=0$.  However, this alignment
of the roots happens precisely when the period lattice~$\Lambda_0$ is
rectangular, with orthogonal basis $w_2, w_1+w_2$, which is when
$\tau$ lies on the semicircle as claimed.  This establishes the first
fact.

Finally, one can check that for $i\in\F^+(4)$ we have
$f(i)=3/\sqrt{2}>0$ (equivalently,
$\lambda(2i)=(3-2\sqrt{2})/(3+2\sqrt{2})$ so that $|\lambda(2i)|<1$).
\end{proof}

\section{Chains of 2-Isogenies}\label{sec:isogeny}
From now on we will use standard Weierstrass models of elliptic curves
rather than the special forms $E_{\{a,b\}}$ used above.  Thus, let
$E_0$ be an elliptic curve over $\C$ given by a Weierstrass equation
\begin{equation}
E_0:\quad Y_0^2 = 4(X_0-e_1^{(0)})(X_0-e_2^{(0)})(X_0-e_3^{(0)}),\label{eqn:E0W}
\end{equation}
where the roots~$e_j^{(0)}$ are distinct, and $\sum_{j=1}^3
e_j^{(0)}=0$.  We consider the ordering of the roots~$e_j^{(0)}$ as
fixed, with the point $T_0=(e_1^{(0)},0)$ of order~$2$ as
distinguished.  Let
\[
  a_0 = \pm\sqrt{e_1^{(0)} - e_3^{(0)}},\quad 
  b_0 = \pm\sqrt{e_1^{(0)} - e_2^{(0)}};
\]
the choice of signs will be discussed below.  Via a shift of the
$X$-coordinate we have $E_0\cong E_{\{a_0,b_0\}}$, and the choice of
signs determines the point
$P_0=(e_1^{(0)}+a_0b_0,2a_0b_0(a_0+b_0))$ of order~$4$ such that
$2P_0=T_0$.

Now consider arbitrary AGM sequences $((a_n,b_n))_{n=0}^\infty$
starting from~$(a_0,b_0)$.  As in \cite{bos-mes}, for $n\ge1$ we let
\begin{equation}\label{eqn:next-ei}
e_1^{(n)} = \frac{a_{n}^2 + b_{n}^2}3,\quad
e_2^{(n)} = \frac{a_{n}^2 - 2b_{n}^2}3,\quad
e_3^{(n)} = \frac{b_{n}^2-2a_{n}^2}3.
\end{equation}
These equalities also hold for $n=0$, and for all~$n\ge0$
the~$e_j^{(n)}$ are distinct, and satisfy $\sum_{j=1}^3
e_j^{(n)}=0$. Hence each AGM sequence determines a sequence
$(E_n)_{n=0}^\infty$ of elliptic curves defined over~$\C$, where $E_n$
is given by the Weierstrass equation
\begin{equation}
E_n:\quad Y_n^2 = 4(X_n-e_1^{(n)})(X_n-e_2^{(n)})(X_n-e_3^{(n)}).\label{eqn:EnW}
\end{equation}
Each has a distinguished $2$-torsion point~$T_n=(e_1^{(n)},0)$.

For $n\ge1$,  define a $2$-isogeny $\varphi_n: E_n\to E_{n-1}$ via
$(x_n,y_n)\mapsto(x_{n-1},y_{n-1})$, where
\begin{equation}\label{eqn:xn-formula}
\begin{split}
x_{n-1} &= x_n + \frac{(e_3^{(n)}-e_1^{(n)})(e_3^{(n)}-e_2^{(n)})}
{x_n - e_3^{(n)}},\\
y_{n-1} &= y_n\left(1 - \frac{(e_3^{(n)}-e_1^{(n)})(e_3^{(n)}-e_2^{(n)})}
{(x_n - e_3^{(n)})^2}\right).
\end{split}
\end{equation}
Now $\ker(\varphi_n)=\ideal{(e_3^{(n)},0)}$, and
\[ 
   \varphi_n(T_n) = T_{n-1} = \varphi_n((e_2^{(n)},0)).
\]
The dual isogeny $\hat{\varphi}_n: E_{n-1}\to E_n$ has kernel
$\ideal{T_{n-1}} \not=\ker(\varphi_{n-1})$, so is distinct
from~$\varphi_{n-1}$.  Each composite~$\varphi_n\circ\varphi_{n+1}$ has
cyclic kernel, since 
$$
\varphi_n\left(\varphi_{n+1}\left(T_{n+1}\right)\right) = 
\varphi_n(T_{n}) = T_{n-1}\ne O.
$$ Similarly, by tracing the images of $T_{n}$, we see that all
composites of the $\varphi_n$ have cyclic kernels.

This \emph{chain of $2$-isogenies} may be depicted thus:
$$
\xymatrix{
\cdots & E_n\ar[l]^<{1}\ar[d]^<<{2}
\ar@<0.8ex>[r]^<{3}|{\varphi_n} & E_{n-1}\ar@<0.8ex>[l]^<{1}|{\hat{\varphi}_n}
\ar[d]^<<{2}\ar[r]^<{3} &
\cdots & E_1\ar[l]^<{1}\ar@<0.5ex>[r]^<{3}\ar[d]^<<{2}
& E_0\ar@<0.5ex>[l]^<{1}\\
 & & & & &}
$$ The number $j$ next to each arrow originating from $E_n$ denotes
the point $(e_j^{(n)},0)$ which generates the kernel of an associated
$2$-isogeny.

The construction of this isogeny chain from the original
curve~$E^{(0)}$ depends on many choices.  The definition of $a_0,b_0$
depends first on which root is labelled~$e_1^{(0)}$ (which
determines~$T_{0}$ and hence~$E_1$), and the order of labelling of
$e_2^{(0)}$ and~$e_3^{(0)}$.  Secondly, the signs for $a_0,b_0$ were
arbitrary; changing just one of them changes~$P_{0}$ and hence~$E_2$.
So, as in the previous section, the unordered pair $\{a_0,b_0\}$
determines the short isogeny chain $E_2\rightarrow E_1\rightarrow
E_0$, with $\{-a_0,-b_0\}$ determining the same short chain.  Finally,
for each~$a_0,b_0$, there are many AGM sequences, which determine the
rest of the chain.

Note that we can rewrite $e^{(n+1)}_j$ given by \eqref{eqn:next-ei} as
\[
e_1^{(n+1)} = \frac{e_1^{(n)} + 2a_{n}b_{n}}{4},\quad
e_2^{(n+1)} = \frac{e_1^{(n)} - 2a_{n}b_{n}}{4},\quad
e_3^{(n+1)} = \frac{-e_1^{(n)}}{2}.
\] 
If $(a_n, b_n)$ is replaced by~$(a_n, -b_n)$ for $n\ge1$, we see that
this interchanges $e_1^{(n+1)}$ and $e_2^{(n+1)}$ but leaves
$e_3^{(n+1)}$ unchanged.  This does not change the curve~$E_{n+1}$; it
only changes the labelling of its roots, which then changes~$E_{n+2}$.

Hence we have established a bijection between
\begin{itemize}
\item The set of all AGM sequences starting at $(a_0,b_0)$, and
\item The set of all isogeny chains starting with the short chain
$E_2\to E_1\to E_0$.
\end{itemize}

We now consider what happens when $n\to\infty$.  From
\eqref{eqn:next-ei}, we have
\begin{equation}\label{eqn:ei-limit}
\lim_{n\to\infty} e_1^{(n)} = \frac{2M^2}{3},\quad 
\lim_{n\to\infty} e_2^{(n)} = \lim_{n\to\infty} e_3^{(n)} = \frac{-M^2}{3},
\end{equation}
where $M=M_S(a_0,b_0)$ for some set $S\subseteq\Z_{>0}$.  The
``limiting curve'' $E_\infty$ for the sequence $(E_n)_{n=0}^\infty$ is
the singular curve
\begin{equation}\label{eqn:limit-curve}
E_\infty: \quad Y_\infty^2 = 4\left(X_\infty - \frac{2}{3}M^2\right)
                             \left(X_\infty + \frac{1}{3}M^2\right)^2.
\end{equation}
Proposition~\ref{prop:cox} implies the following.
\begin{proposition}
The singular point of~$E_{\infty}$ is a node if and only if the AGM
sequence $(a_n,b_n)$ is good.
\end{proposition}

\subsection{The associated lattice chain}

For each $n\ge0$, we have $E_n(\C)\cong\C/\Lambda_n$, where
$\Lambda_n$ is the lattice of periods of the
differential~${dX_n}/{Y_n}$.  From the definition of $\varphi_n$ (see
\eqref{eqn:xn-formula}), it can be verified that each~$\varphi_n$ is
\emph{normalised}, in the sense that
\begin{equation}\label{eqn:diff}
\varphi_{n}^*\left(\frac{dX_{n-1}}{Y_{n-1}}\right) =
\frac{dX_{n}}{Y_{n}}
\end{equation}
for all $n\ge1$.  Hence $\varphi_n$ corresponds to the map
$\C/\Lambda_n\to\C/\Lambda_{n-1}$ induced from the identity map
$\C\to\C$.  Since each $\varphi_n$ is a $2$-isogeny and the composites
of the $\varphi_n$ have cyclic kernels, it is clear that
$(\Lambda_{n})$ is a lattice chain in the sense of
Section~\ref{sec:lattice-chain}.

This establishes the commutativity of the diagram (Figure 1), which
shows the relationship between chains of lattices and chains of
$2$-isogenies.  For brevity we denote $z\mapsto (\wp_\Lambda(z),
\wp'_\Lambda(z))$ by~$z\mapsto\wp_n(z)$.

\begin{figure}[htp]
$$ \xymatrix{ \cdots\ar[r] & \C\ar[r]^{\mathrm{id}}\ar[d] &
    \C\ar[d]\ar[r] & \cdots\\ \cdots\ar[r] &
    \C/\Lambda_n\ar[r]\ar[d]^{\wp_n} &
    \C/\Lambda_{n-1}\ar[d]^{\wp_{n-1}} \ar[r]
    & \cdots\\ \cdots\ar[r] & E_n\ar[r]^{\varphi_n} & E_{n-1}\ar[r] &
    \cdots }
$$
\caption{A chain of isogenies linked with a chain of
  lattices}\label{fig:isogenies} 
\end{figure}

Conversely, given any lattice chain $(\Lambda_{n})$ starting
from~$\Lambda_0$, we may recover the sequence of curves~$E_n$ and the
chain of $2$-isogenies linking them: first, $\Lambda_1$ determines
which root of~$E_0$ is labelled~$e_1^{(0)}$; then $\Lambda_2$
determines the choice of signs in the definition of~$a_0,b_0$; and
finally the AGM sequence starting with $(a_0,b_0)$ is determined by
the $\Lambda_n$ for $n\ge2$.

Thus we have a third set in bijection with both the set of all AGM
sequences starting at $(a_0,b_0)$, and the set of all isogeny chains
starting with the short chain $E_2\to E_1\to E_0$: namely, the set of
all lattice chains starting with the short chain
$\Lambda_0\supset\Lambda_1\supset\Lambda_2$.

\begin{proposition} \label{prop:good=good}
With the above notation, for all~$n\ge0$,
\begin{enumerate}
\item $E_n\cong E_{\{a_n,b_n\}}$;
\item $\Lambda_{n}\supset\Lambda_{n+1}\supset\Lambda_{n+2}$ is a short
  chain in the sense of Section~3;
\item $\Lambda_{n+2}$ is the right choice of sublattice of
  $\Lambda_{n+1}$ if and only if $(a_n,b_n)$ is a good pair;
\item the lattice chain $(\Lambda_n)$ is good (respectively, optimal)
  if and only if the sequence $((a_n,b_n))$ is good  (respectively,
  optimal). 
\end{enumerate}
\end{proposition}
\begin{proof}
For (1), replace~$x_n$ by $x_n+e_1^{(n)}$ in the equation for $E_n$ to
obtain the equation for~$E_{\{a_n,b_n\}}$.  The rest is then is clear,
using Theorem~\ref{prop:good=right}.
\end{proof}

\section{Period Lattices of Elliptic Curves}\label{sec:periods}
\subsection{General Case}
Let $E_0$ be an elliptic curve over $\C$ of the form~(\ref{eqn:E0W}).
We keep the notation of the preceding section; in particular, the
period lattice of~$E_0$ is~$\Lambda_0$.  Each primitive
period~$w_1\in\Lambda_0$ determines a good lattice chain $(\Lambda_n)$
where $\Lambda_n=\ideal{w_1}+2^n\Lambda_0$, and conversely, since
$\cap_n\Lambda_n=\ideal{w_1}$.  So we have a bijection between the set
of primitive periods of~$\Lambda_0$ (up to sign) and good lattice
chains.  Each good lattice chain in turn determines a good AGM
sequence $((a_n,b_n))$ starting at a pair $(a_0,b_0)$ such that
$E_0\cong E_{\{a_0,b_0\}}$.

We now show that the period~$w_1$ may be expressed simply in terms of
the limit of the associated AGM sequence.  It will follow that every
primitive period~$w_1$ of $E_0$ may be obtained from the limit of an
appropriately chosen good AGM sequence.  Conversely, we may express
the set of all limits of AGM sequences starting at $(a_0,b_0)$ in
terms of periods of~$E_0$.  We will also show that optimal AGM
sequences give periods which are minimal in their coset
modulo~$4\Lambda_0$, and super-optimal sequences (where the initial
pair~$(a_0,b_0)$ also good) give periods which are minimal
modulo~$2\Lambda_0$. By Lemma~\ref{lem:z-basis}, we will be then able
to express a $\Z$-basis for~$\Lambda_0$ in terms of specific AGM
values.

\bigskip

\begin{proposition}\label{prop:wp-limits}
Let $(\Lambda_n)$ be a good lattice sequence with limiting period $w_1$
(generating $\cap\Lambda_n$, and defined up to sign).  Then for all
$z\in\C\setminus\Lambda_0$ we have
\[
\begin{aligned}
\lim_{n\to\infty} \wp_{\Lambda_n}(z) =
\left(\frac{\pi}{w_1}\right)^2\left(\frac{1}{\sin^2(z\pi/w_1)}-\frac{1}{3}\right) 
\\
\lim_{n\to\infty} \wp'_{\Lambda_n}(z) =
-2\left(\frac{\pi}{w_1}\right)^3\left(\frac{\cos(z\pi/w_1)}{\sin^3(z\pi/w_1)}\right).
\end{aligned}
\]
\end{proposition}

\begin{proof}
Since $w_1$ is primitive, there exists $w_2\in\C$ such that
$\Lambda_n=\ideal{w_1,2^nw_2}$ for all~$n\ge0$.  In the
standard series expansion
\[
   \wp_{\Lambda_n}(z) = \frac{1}{z^2} + \sum_{0\not=w\in\Lambda_n}\left(\frac{1}{(z-w)^2}-\frac{1}{w^2}\right),
\]
we set $w=m_1w_1+m_22^nw_2$ with $m_1,m_2$ not both zero.  As
$n\to\infty$ all terms with $m_2\not=0$ tend to zero, leaving
\[
  \lim_{n\to\infty} \wp_{\Lambda_n}(z) =
  \sum_{m\in\Z}\frac{1}{(z-mw_1)^2} - \frac{1}{3}\left(\frac{\pi}{w_1}\right)^2.
\]
Using the expansion $\pi^2/\sin^2(\pi z) = \sum_{m\in\Z}1/(z-m)^2$,
this simplifies to the formula given.

For $\lim_{n\to\infty} \wp'_{\Lambda_n}(z)$, we may either
differentiate this, or apply the same argument to the series expansion
of~$\wp'_{\Lambda_n}(z)$.
\end{proof}

\begin{corollary}\label{cor:one-period}
In the above notation, let $(\Lambda_n)$ be a (good) lattice chain,
with limiting period~$w_1$, associated to the elliptic curve~$E_0$ and
the (good) AGM sequence $((a_n,b_n))$ with non-zero limit
$M=M_S(a_0,b_0)$.  Then $M = \pm \pi/w_1$, so that the period~$w_1$
may be determined up to sign by
\[
    w_1 = \pm \pi/M_S(a_0,b_0).
\]
\end{corollary}

\begin{proof}
For all~$n\ge0$ we have $\wp_{\Lambda_n}(w_1/2) = e_1^{(n)}$.  Letting
$n\to\infty$ and using the proposition we find that 
\[
  \frac{2}{3}M^2 = \lim_{n\to\infty}e_1^{(n)} =
  \frac{2}{3}\left(\frac{\pi}{w_1}\right)^2, 
\]
from which the result follows.
\end{proof}

The ambiguity of sign in this result will not matter in practice:
changing the sign of~$w_1$ does not change the lattice chain, and
neither does changing the signs of both~$a_0,b_0$ (and hence the sign
of~$M_S(a_0,b_0)$).

\bigskip

For fixed $(a_0,b_0)$, the value of $M_S(a_0,b_0)$ depends on the
set~$S$ of indices~$n$ for which~$(a_n,b_n)$ is bad.  Changing~$S$, we
obtain different AGM sequences, and different lattice chains, but
these all start with the same short chain $(\Lambda_n)_{n=0}^{2}$, and
the periods given by $\pi/M_S(a_0,b_0)$ are all in the same coset
modulo~$4\Lambda_0$.  We may now establish the result stated above as
Proposition~\ref{prop:cox}(3):

\begin{corollary}
$|M_S(a_0,b_0)|$ attains its maximum (among all AGM-sequences starting
  at $(a_0,b_0)$) if and only if the sequence is optimal.
\end{corollary}
\begin{proof}
By Corollary~\ref{cor:one-period}, $M_S(a_0,b_0)$ is maximal (in
absolute value) if and only if the limiting
period~$w_1=\pi/M_S(a_0,b_0)$ is minimal.  By
Proposition~\ref{prop:min-coset-rep}, this is if and only if the
lattice chain is optimal.  By Proposition~\ref{prop:good=good}(4),
this in turn is if and only if the AGM sequence is optimal.
\end{proof}

\begin{corollary}
\begin{enumerate}
\item The optimal value $M=M(a_0,b_0)$ gives a period $w_1=\pi/M$
  which is minimal in its coset modulo~$4\Lambda_0$.
\item $|M(a_0,b_0)| \ge |M(a_0,-b_0)| \iff |a_0-b_0|\le|a_0+b_0|$.
\item If $(a_0,b_0)$ is good, then $\pi/M(a_0,b_0)$ is minimal in its
  coset modulo~$2\Lambda_0$.
\end{enumerate}
\end{corollary}

\begin{proof}
Changing the sign of~$b_0$ (only) does not change $\Lambda_1$ (or
$E_1$), but does change~$\Lambda_2$.  The effect on~$w_1$, therefore,
is to change its coset modulo~$4\Lambda_0$ while not affecting its
coset modulo~$2\Lambda_0$.  By Proposition~\ref{prop:good=right},
$\Lambda_2$ is the right choice if and only if~$(a_0,b_0)$ is good.
Hence, to obtain a period minimal in its coset modulo~$2\Lambda_0$,
and not just modulo~$4\Lambda_0$, we choose the sign of~$b_0$ so that
the pair~$(a_0,b_0)$ is good, and then take an optimal AGM sequence.
\end{proof}

\begin{theorem}[Periods of Elliptic Curves over $\C$, first
    version]\label{thm:w1} 
Let $E$ be an elliptic curve over $\C$ given by the Weierstrass
equation
\[
    Y^2 = 4(X-e_1)(X-e_2)(X-e_3),
\] 
with period lattice~$\Lambda$. Set $a_0 = \sqrt{e_1-e_3}$ and~$b_0 =
\sqrt{e_1-e_2}$, where the signs are chosen so that $(a_0,b_0)$ is
good ({\it i.e.}, $|a_0-b_0| \le |a_0+b_0|$), and let
\[
    w_1 = \frac{\pi}{M(a_0,b_0)},
\] 
using the optimal value of the AGM.  Then $w_1$ is a primitive period
of~$E$, and is a minimal period in its coset modulo~$2\Lambda$.

Define $w_2$, $w_3$ similarly by permuting the $e_j$;  then any two of
$w_1,w_2,w_3$ form a $\Z$-basis for~$\Lambda$.
\end{theorem}

\begin{proof}
Everything has been established except the last part.  Letting
$e_2,e_3$ in turn play the role of~$e_1$ gives minimal periods in each
of the cosets modulo~$2\Lambda$, so Lemma~\ref{lem:z-basis} applies.
\end{proof}

\begin{algorithm}[Computation of a period lattice basis]\label{algo:zbasis} 
\hfill\break\vspace{-10pt}
\begin{itemize}
\item[\textbf{Input:}] An elliptic curve $E$ defined over~$\C$, and roots
$e_j\in\C$ for $j=1,2,3$.  
\item[\textbf{Output:}] Three primitive periods of~$E$, which are minimal
coset representatives, any two of which form a $\Z$-basis for the
period lattice of $E$.
\end{itemize}
\begin{enumerate}
\item Label one of the roots as $e_1$, and the other two
  arbitrarily as~$e_2$, $e_3$;
\item Set $a_0 = \sqrt{e_1 - e_3}$ with arbitrary sign,
  and then $b_0 = \pm\sqrt{e_1 - e_2}$ with the sign
  chosen such that $ |a_0-b_0|\le|a_0+b_0|$.
\item Output $w = \pi/M(a_0,b_0)$, using the optimal value of the AGM.
\item Repeat with each root~$e_j$ in turn playing the role
  of~$e_1$.
\end{enumerate}
\end{algorithm}

Instead of computing $w_2,w_3$ by permuting the~$e_j$ as in
Theorem~\ref{thm:w1}, we may alternatively obtain all $w_j$ by using a
single ordering of the roots and three different AGM computations.

Starting with an arbitrary ordering of the roots,say $(e_1, e_2,
e_3)$, define $a$ and~$b$ as before, up to sign, by $a^2=e_1-e_3$ and
$b^2=e_1-e_2$; and also define $c$ (up to sign) by $c^2=e_2-e_3$, so
that $a^2 = b^2 + c^2$. We would like to determine the signs of
$a,b,c$ so that all three of the following conditions hold:
\begin{equation}\label{eqn:abc-cond}
|a-b|\le |a+b|,\quad
|c-ib| \le |c+ib|,\quad
|a-c| \le |a+c|.
\end{equation}
We claim that this is always possible. To see this, first choose the
sign of $a$ arbitrarily. Then choose the signs of $b$ and $c$ so that
the first and the third conditions in \eqref{eqn:abc-cond} hold.
Finally, if the second condition fails, one can easily check that if
$e_1$ and $e_3$ are interchanged and $a,b,c$ replaced (in order)
by~$ia$, $ic$, $ib$, then all three inequalities will hold.

We can now state an alternative theorem for obtaining a $\Z$-basis for
the period lattice~$\Lambda$ of~$E$.

\begin{theorem}[Periods of Elliptic Curves over $\C$, second
    version]\label{thm:w1w2w3} 
Let $E$ be an elliptic curve over $\C$ given by the Weierstrass
equation
\[
    Y^2 = 4(X-e_1)(X-e_2)(X-e_3),
\] 
with period lattice~$\Lambda$. Order the roots~$(e_1, e_2, e_3)$
of~$E$, so that the signs of $a=\sqrt{e_1-e_3}$, $b=\sqrt{e_1-e_2}$,
$c=\sqrt{e_2-e_3}$ may be chosen to satisfy all the conditions of
\eqref{eqn:abc-cond}. Define
$$
w_1 = \frac\pi{M(a,b)},\quad
w_2 = \frac\pi{M(c, ib)},\quad
w_3 = \frac{i\pi}{M(a,c)}.
$$ Then each~$w_j$ is a primitive period, minimal in its coset
modulo~$2\Lambda$, and any two of the $w_j$ form a $\Z$-basis for
$\Lambda$.
\end{theorem}

\begin{proof}
Let $(e_1, e_2, e_3)$ be an order of the roots of $E_0$.
Interchanging~$e_1$ and~$e_3$ if necessary, define $a=\sqrt{e_1-e_3}$,
$b=\sqrt{e_1-e_2}$, $c=\sqrt{e_2-e_3}$, with the signs chosen so that
all three inequalities in \eqref{eqn:abc-cond} hold.

Now $w_1=\pi/M(a,b)$ is primitive and minimal in its coset as before,
since $(a,b)$ is good.  Using $(e_1',e_2',e_3')=(e_2,e_1,e_3)$, we
find that $(a',b')=(c,ib)$ is good, and set $w_2 = \pi/M(a',b') =
\pi/M(c,ib)$; and using $(e_1'',e_2'',e_3'')=(e_3,e_2,e_1)$, we see
that $(a'',b'')=(ia,ic)$ is good, and set $w_3= \pi/M(a'',b'')=\pi
i/M(a,c)$.
\end{proof}

\bigskip
We complete this section by considering two special cases, which arise
when considering elliptic curves defined over the real numbers,
separating the cases of positive discriminant (rectangular period
lattice) and negative discriminant.

\subsection{Special Case I: Rectangular Lattices}\label{sec:special-i}
Recall that if $|a_0-b_0|=|a_0+b_0|$, then both $(a_0,\pm b_0)$ are good and
$\Re(b_0/a_0)=0$. In this case,
$$
   \frac{e_2-e_1}{e_3-e_1} = (b_0/a_0)^2
$$ is real and negative.  Geometrically, this means that the~$e_j$ are
   collinear on the complex plane with $e_1$ in the middle.

To see what the associated period lattice looks like, let $w =
\pi/M(a_0,b_0)$ and $w'=\pi/M(a_0,-b_0)$. Then $w,w'$ are both minimal
elements in the same coset modulo $2\Lambda_0$.  By Lemma
\ref{lem:ortho-basis}, the periods~$w_1,w_2=(w\pm w')/2$ form an
orthogonal $\Z$-basis for $\Lambda_0$, and the period lattice is
rectangular.  Alternatively, we could obtain a $\Z$-basis for
$\Lambda_0$ by computing two periods (as in Theorem~\ref{thm:w1})
using the two other roots of $E$ which are not ``in the middle'' in
the role of~$e_1$.

Finally, we note that whenever the $e_j$ are collinear, we can
``rotate'' them by a multiplying by a suitable constant in $\C^*$ so
that the scaled roots $e'_j$ are all real. Then one could use an
algorithm for computing period lattices of elliptic curves over $\R$
(e.g.  \cite[Algorithm 7.4.7]{coh}) to compute the period lattice of
the elliptic curve $(Y')^2 = 4(X'-{e'_1})(X'-e'_2)(X'-e'_3)$.  The
period lattice of our original elliptic curve is then obtained after
suitable scaling.  This may be more efficient in practice, since only
real arithmetic would be needed in the AGM iteration.

If the $e_j$ are all real (as is the case for an elliptic curve
defined over~$\R$ with positive discriminant), we may order them so
that $e_1>e_2>e_3$ and obtain a rectangular basis for the period
lattice by setting
\begin{equation}\label{eqn:pos-real-periods}
    w_1 = \pi/M(\sqrt{e_1-e_2},\sqrt{e_1-e_3}), \quad
    w_2 = \pi i/M(\sqrt{e_2-e_3},\sqrt{e_1-e_3})
\end{equation}
with all square roots positive; then $w_1$ and $w_2/i$ are both real
and positive.  These familiar formulas may be found in \cite[Algorithm
  7.4.7]{coh} or \cite[(3.7.1)]{JCbook2}.

\subsection{Special Case II}\label{sec:special-ii}
If the roots of $E$ are such that
$$
\left|\frac{e_1-e_2}{e_1-e_3} \right|=1\quad
\text{with $e_1-e_2 \ne \pm(e_1-e_3)$},
$$ then geometrically the $e_j$ lie on an isosceles triangle having
$e_1$ as the vertex where the sides of equal length intersect. As
before, one can rotate this triangle by a suitable constant in $\C^*$
so that $e_1\in\R$, and $e_2,e_3$ are complex conjugates. This yields
a new elliptic curve $E'$, defined over $\R$, whose Weierstrass
equation has only one real root.

Again, one could use an algorithm for computing period lattices of
elliptic curves over $\R$ (e.g. \cite[Algorithm 7.4.7]{coh}) to
compute the period lattice of $E'$.  This is of the form
$\Lambda'=\ideal{w'_1,w'_2}$, for some $w'_1,w'_2$ satisfying
$$
w'_1\in\R,\quad \Re(w'_2) = \frac{w'_1}2.
$$ The period lattice $\Lambda=\ideal{w_1,w_2}$ of $E$, with
$\Re(w_2/w_1)=1/2$, can then be obtained by a suitable scaling of
$w'_1,w'_2$.  This will be illustrated in Example \ref{eq:special-ii}.

For real curves with negative discriminant, we present here a
simplification of the purely real algorithm given in~\cite{coh}.  Let
$e_1$ be real and $e_2,e_3$ complex conjugates, ordered so that $\Im
e_2>0$.  Set $a_0=\sqrt{e_1-e_3}=x+yi$; since $e_1-e_3$ lies in the
upper half-plane, we may choose the sign of~$a_0$ so that~$x,y>0$.
Set $r=\sqrt{x^2+y^2}>0$ and $b_0=\sqrt{e_1-e_2}=x-yi$.  Now we may
obtain a real period~$w_+$ from
\[
   w_+ = \pi/M(a_0,b_0) = \pi/M(x+yi,x-yi) = \pi/M(x,r),
\]
and an imaginary period~$w_-$ from
\[
   w_- = \pi/M(-a_0,b_0) = \pi i/M(y-xi,y+xi) = \pi i/M(y,r).
\]
Note that both AGMs appearing here, $M(x,r)$ and~$M(y,r)$, are
classical (real and positive). These periods span a sublattice of
index~$2$ in the period lattice, for which a $\Z$-basis may be taken
to be~$w_1 = w_+$ and~$w_2 = (w_++w_-)/2$, where $\Re(w_2/w_1)=1/2$.

\section{The complete set of AGM values}
In 1800, Gauss described the complete set of values of $M_S(a,b)$ as
$S$ ranges through all finite sets.  The proof given by Cox
in~\cite[Theorem 2.2]{cox} uses theta and modular functions related to
the modular functions which appeared earlier in this paper.  Other
proofs are also available in the literature, for example by Geppert
\cite{gep}.

We will give here a slightly more general form of the result than that
stated in \cite{cox}, and give an alternative proof which brings out
clearly the relation with period lattices of elliptic curves.  

In the following statement, we set $P_S(a,b)=\pi/M_S(a,b)$ (for
any finite $S\subseteq\Z_{>0}$) and $P(a,b)=\pi/M(a,b)$.
\begin{theorem}
For $a,b\in\C^*$ with $a\ne\pm b$, let $E_{\{a,b\}}$ be the elliptic curve
over $\C$ given by the Weierstrass equation
\[
   E_{\{a,b\}}:\quad Y^2 = 4X(X+a^2)(X+b^2),
\] 
and let $\Lambda$ be its period lattice. Let $c=\sqrt{a^2-b^2}$, with
the sign chosen so that the pair $(a,c)$ is good, and set
\[
   w_1=P(a,b),\quad w_3=iP(a,c).
\] 
Then $\Lambda = \Z w_1+\Z w_3$, and the set of values of $P_S(a,b)$ is
precisely the set of primitive elements of the coset $w_1+4\Lambda$.
More precisely, we have the following:
\begin{align*}
   \{P_S(a,b)\} &= \{ w \in w_1 + 4\Lambda, \quad\text{$w$ primitive}\}; \\
   \{P_S(a,-b)\} &= \{ w \in w_1 + 2w_3 + 4\Lambda, \quad\text{$w$ primitive}\}; \\
   \{P_S(-a,-b)\} &= \{ w \in -w_1 + 4\Lambda, \quad\text{$w$ primitive}\}; \\
   \{P_S(-a,b)\} &= \{ w \in -w_1  + 2w_3 + 4\Lambda, \quad\text{$w$ primitive}\}.
\end{align*}
Thus, the complete set of all values of $P_S(\pm a, \pm b)$ is the set
of primitive elements of the coset $w_1+2\Lambda$.
\end{theorem}

\begin{proof}
Since $\Lambda$ is invariant under translations of the $X$-coordinate,
we may apply Theorem \ref{thm:w1w2w3} to see that $\Lambda=\Z w_1'+\Z
w_3$ where $w_3$ (as given) is a minimal coset representative, and either
\begin{itemize}
\item {$(a,b)$ is good} and $w_1'=w_1$; or
\item {$(a,b)$ is bad} and  $w_1'=w_1\pm 2w_3$.
\end{itemize}
In either case, $\Lambda=\Z w_1 + \Z w_3$.

Now the values of $P_S(a,b)$ are precisely the primitive periods in
the same coset as~$w_1=P(a,b)$ modulo~$4\Lambda$.  Secondly,
$P_S(-a,-b)=-P_S(a,b)=-w_1$, so the values of $P_S(-a,-b)$ are the
primitive periods in the coset $-w_1\pmod{4\Lambda}$, as
required. Next, $P(a,-b)$ is the minimal period in the coset
$w_1+2w_3+4\Lambda$, since this is the other coset modulo~$4\Lambda$
contained in $w_1+2\Lambda$, so the values of~$\pm P_S(a,-b)$ are also
as stated.
\end{proof}

\begin{corollary}
Let $a,b,c\in\C^*$ satisfy $a^2=b^2+c^2$.  Define $w=\pi/M(a,b)$ and
$w'=\pi i/M(a,c)$, where $(a,c)$ is a good pair.  Then 
\begin{enumerate}
\item $\Lambda=\Z w+\Z w'$ is a lattice in $\C$;
\item the set of values of $\pi/M_S(a,b)$
is the set of primitive elements of the coset $w+4\Lambda$;
that is, the set 
\[
  \{uw+vw'\mid u,v\in\Z,\  \gcd(u,v)=1,\ u-1\equiv v\equiv0\pmod{4}\};
\] 
\item the set of values of $\pi/M_S(\pm a,\pm b)$
is the set of primitive elements of the coset $w+2\Lambda$;
that is, the set 
\[
  \{uw+vw'\mid u,v\in\Z,\  \gcd(u,v)=1,\ u-1\equiv v\equiv0\pmod{2}\}.
\] 
\end{enumerate}
\end{corollary}

\section{Elliptic Logarithms} \label{sec:elog}
We now extend the method for computing periods of elliptic curves in
Section \ref{sec:periods} to give a method for computing elliptic
logarithms of points on elliptic curves. 

Let $E$ be an elliptic curve over $\C$ given by a Weierstrass equation
as before, and $\Lambda$ the lattice of periods of the differential
$dX/Y$ on~$E$, so that $E(\C)\cong\C/\Lambda$. An \emph{elliptic
  logarithm} of $P\in E(\C)$ is a value $z_P\in\C$ such that
$P=(\wp_{\Lambda}(z_P), \wp'_{\Lambda}(z_P))$. Note that $z_P$ is only
well-defined modulo $\Lambda$.  We wish to have an algorithm which can
compute the numerical value of the complex number~$z_P$, to any
required precision, from the coefficients of~$E$ and the coordinates
of~$P$ (which we assume are given exactly, or are available to
arbitrary precision).

Construct as before an isogeny chain $(E_n)$ with $E_0=E$, with
associated lattice chain $(\Lambda_n)$ (with $\Lambda_0=\Lambda$) and
AGM sequence $(a_n,b_n)$.  We will assume that the chain is
super-optimal with $|a_n-b_n| < |a_n+b_n|$ for all~$n\ge0$.  (This is
possible except when~$\Lambda_0$ is rectangular, and even then is
possible for two of the three super-optimal sequences).  Let $w_1,w_2$
be a $\Z$-basis for~$\Lambda$ such that
$\Lambda_n=\left<w_1,2^nw_2\right>$ for all~$n\ge0$.  We have
$2$-isogenies $\varphi_n:E_n\to E_{n-1}$ for~$n\ge1$, induced by the
natural maps $\C/\Lambda_n\to\C/\Lambda_{n-1}$.

\subsection{Coherent point sequences}

Consider sequences of points~$(P_n)_{n=0}^{\infty}$ where~$P_n\in
E_n(\C)$ satisfy $\varphi_n(P_n)=P_{n-1}$ for all~$n\ge1$.  Such a
sequence will be called \emph{coherent} if there exists $z\in\C$ such
that $P_n = \wp_n(z)$ for all~$n\ge0$; here, as above, we write
$\wp_n(z)$ for~$(\wp_{\Lambda_n}(z), \wp'_{\Lambda_n}(z))$.  If
such a~$z$ exists, it is uniquely determined
modulo~$\cap\Lambda_n=\Lambda_{\infty}=\left<w_1\right>$.

In general there are uncountably many point sequences with a fixed
starting point~$P_0$, since for each~$P_n\in E_n(\C)$ there are two
points~$P_{n+1}\in E_{n+1}(\C)$ with $\varphi_{n+1}(P_{n+1})=P_{n}$.
However, only countably many of these are coherent, since
$\wp_0^{-1}(P_0)$ is a coset of~$\Lambda_0$ in~$\C$, and hence
countable.

For example, taking $z=0$ shows that the trivial sequence $(O_n)$,
where $O_n$ is the base point on~$E_n$, is coherent.  Also, the
sequence with $P_n=T_n=(e_1^{(n)},0)$ is coherent, via $z=w_1/2$.

Given a point sequence $(P_n)$, for each~$n$ let
$C_n=\wp_n^{-1}(P_n)\subset\C$ be the complete set of all the elliptic
logarithms of~$P_n$, which is a coset of $\Lambda_n$ in~$\C$.  Since
$\Lambda_{n+1}$ has index~$2$ in~$\Lambda_n$, each $C_n$ is the
disjoint union of two cosets of~$\Lambda_{n+1}$, one of these being
$C_{n+1}$; the other is the set of elliptic logarithms of the second
point~$P_{n+1}'\in E_{n+1}(\C)$ such that
$\varphi_{n+1}(P_{n+1}')=P_n$.  Thus we have
\[
   C_0 \supset C_1 \supset \dots \supset C_n \supset C_{n+1} \supset \dots.
\]
The point sequence is coherent if and only if
$C_{\infty}=\cap_{n=0}^{\infty}C_n\not=\emptyset$,  in which case
$C_{\infty}$ is a coset of $\Lambda_{\infty}$ in~$\C$.

An argument similar to that used above for Lemma~\ref{lem:minimal}
shows the following.

\begin{lemma}
The sequence $(P_n)$ is coherent if and only if $C_{n+1}$ contains the
smallest element of~$C_n$ for almost all~$n\ge0$.
\end{lemma}

\subsection{The elliptic logarithm formula}

\begin{proposition}\label{prop:elog-formula}
With notation as above, let $(P_n)$ be a coherent point sequence
determined by $z\in\C$.  Assume that $2z\not\in\Lambda_{\infty}$. Then
for~$n$ sufficiently large, we have $P_n\not= O_n$, and
write~$P_n=(x_n,y_n)$.  Let $P_{\infty}=(x_{\infty},y_{\infty})\in
E_{\infty}(\C)$ be the limit point, defined by
$(x_{\infty},y_{\infty}) = \lim_{n\to\infty}(x_n,y_n)$.  Set
$M=\pi/w_1$, and
\[
    t_{\infty} = -\frac{1}{2}y_{\infty}/(x_{\infty}+M^2/3).
\]  
Then $t_{\infty}\not=0,\infty$, and (modulo~$\Lambda_{\infty}$) we have
\begin{equation}
    z = \frac{1}{M}\arctan\left(\frac{M}{t_{\infty}}\right) 
      = \frac{w_1}{\pi}\arctan\left(\frac{\pi}{w_1t_{\infty}}\right).
\label{eqn:elog-formula}
\end{equation}
\end{proposition}

\begin{proof}
Since $z\not\in\Lambda_{\infty}$, for all~$n\gg0$ we have
$z\not\in\Lambda_n$, so that $P_n\not= O_n$.
Proposition~\ref{prop:wp-limits} gives expressions for the coordinates
of~$P_{\infty}=(x_{\infty},y_{\infty})\in E_{\infty}(\C)$ in terms of
$M$, $s=\sin(z\pi/w_1)$ and~$c=\cos(z\pi/w_1)$:
\[
x_{\infty} = M^2\left(\frac{1}{s^2}-\frac{1}{3}\right); \qquad
y_{\infty} = -2M^3\frac{c}{s^3}.
\]
Note that $s\not=0$, since $z\not\in\Lambda_{\infty}$; also,
$s\not=\pm1$ (and $c\not=0$) since $2z\not\in\Lambda_{\infty}$.  Thus
$x_{\infty}+M^2/3 = M^2/s^2 \not= 0$, and $t_{\infty} =
-\frac{1}{2}y_{\infty}/(x_{\infty}+M^2/3) = Mc/s\not=0$, giving 
formula~(\ref{eqn:elog-formula}).  Taking different values of the
multiple-valued function~$\arctan$ changes~$z$ by integer multiples
of~$w_1$; so this formula gives a well defined value for $z$
modulo~$\Lambda_{\infty}$, as desired.
\end{proof}

This result does also apply when $z=\pm w_1/2\pmod{\Lambda_{\infty}}$,
for then $s=\pm1$ and~$c=0$, so $x_{\infty}+M^2/3=M^2$ and
$y_{\infty}=0$, giving $t_{\infty}=0$ and $z=w_1$; this is the case we
used above to compute periods.

Proposition~\ref{prop:elog-formula}, and in particular
formula~(\ref{eqn:elog-formula}), is the key to our elliptic logarithm
algorithm, in which we will compute a sequence~$(t_n)$ iteratively
such that $\lim t_n=t_{\infty}$.  However, we derived~(12) by starting
from a value of~$z\in\C$, rather than from the coordinates of a
point~$P=\wp(z)\in E(\C)$.  In order to produce an algorithm for
computing~$z$ from the coordinates of~$P$, we must show how to
construct inductively a suitable coherent sequence of points, so that
the limits $x_{\infty}$, $y_{\infty}$ and~$t_{\infty}$ exist. We will
do this in the next subsection.

\begin{remark}
Our formula (\ref{eqn:elog-formula}) is similar to the one used in
Cohen's algorithm \cite[Algorithm 7.4.8]{coh} for computing elliptic
logarithms of real points on elliptic curves defined over~$\R$.  The
variable denoted $c_n$ in \cite{coh} is related to our~$t_n$ (defined
below) by $c_n^2=t_n^2+a_n^2$; setting
$c_{\infty}=\lim_{n\to\infty}c_n$, so that $c_{\infty}^2 =
t_{\infty}^2 + M^2$, we can rewrite $z_P$ as
\[
   z_P = \pm \frac{1}{M}\arcsin\left(\frac{M}{c_{\infty}}\right),
\] 
which is similar (up to sign) to the output of Cohen's algorithm.
This approach leaves an ambiguity of the sign of $z_P$, which is
resolved in \cite{coh} by considering the sign of $y_0$ at the end,
something which is only possible in the real case. Using $t_{\infty}$
instead of~$c_{\infty}$ avoids the ambiguity.
\end{remark}

\subsection{The elliptic logarithm iteration}

Let $P=(x,y)\in E(\C)$, where as above $E$ is the elliptic curve with
equation
\[
  E:\qquad   Y^2 = 4(X-e_1)(X-e_2)(X-e_3).
\]  
In order to compute the elliptic logarithm~$z_P$ of~$P$ using
(\ref{eqn:elog-formula}), we need to find a suitable coherent point
sequence $(P_n)$ starting at~$P_0=P$.  We iteratively compute $P_1$,
$P_2$, $\dots$, using the explicit formulas for the isogenies
$\varphi_n$; at each stage there are two possible choices for~$P_n$,
determined by choosing a specific sign for a square root.  The main
issue is how to make these choices in such a way that the sequences
converge.

It is simpler in practice to use alternative models for the elliptic
curves in the sequence, in which the isogeny formulas are simpler.  We
introduce these now.  Let $E_1'$ be the curve with equation
\[
   E_1':\qquad R^2 = (T^2+a^2)/(T^2+b^2).
\]
We regard $E_1'$ as a projective curve in~$\P^1\times\P^1$, with
points at infinity given by $(t,r)=(\infty,\pm1), (\pm bi,\infty)$.
 
Define a map $\alpha:E_1'\to E$ by\footnote{The sign of $y$ here is
  chosen to avoid a minus sign in the elliptic logarithm formula
  (\ref{eqn:elog-formula}).}
$(t,r)\mapsto(x,y)=(t^2+e_1,-2rt(t^2+b^2))$, where as usual
$a^2=e_1-e_3$ and $b^2=e_1-e_2$.  This map is unramified and
has~degree~$2$; it sends $(\infty,\pm1)\mapsto O_E$, $(\pm bi,
\infty)\mapsto(e_2,0)$, $(0,\pm a/b)\mapsto(e_1,0)$ and $(\pm
ai,0)\mapsto(e_3,0)$.

Write $a_1,b_1$ for the arithmetic and geometric means of $a,b$ as
usual,  set 
\begin{align*}
e_1' &= (a_1^2 + b_1^2)/3 = (a^2+6ab+b^2)/12,\\ 
e_2' &= (a_1^2 -2 b_1^2)/3 = (a^2-6ab+b^2)/12,\\ 
e_3' &= (b_1^2 -2 a_1^2)/3 = -(a^2+b^2)/6,
\end{align*}
so that $E_1$ has with Weierstrass equation
\[
   E_1:\qquad   Y_1^2 = 4(X_1-e_1')(X_1-e_2')(X_1-e_3').
\]
Now $E_1'\cong E_1$ via the isomorphism~$\theta$ given by $(t,r)
\mapsto (x_1,y_1)$ where
\[
(x_1,y_1) = (\frac{1}{2}(t^2+r(t^2+a^2)+\frac{1}{6}(a^2+b^2)),
             t(t^2+r(t^2+a^2)+\frac{1}{2}(a^2+b^2)),
\]
with inverse
\[
(x_1,y_1) \mapsto (t,r) = \left(
              \frac{3y_1}{6x_1+a^2+b^2},
              \frac{12x_1+5a^2-b^2}{12x_1+5b^2-a^2}
        \right).
\]
The composite~$\alpha\circ\theta^{-1}:E_1\to E_1'\to E$ is the
$2$-isogeny denoted~$\varphi$ in Section~\ref{sec:isogeny}.

Given a complete $2$-isogeny chain $(E_n)_{n\ge0}$ with $E_0=E$, as in
Section~\ref{sec:isogeny}, we define for each~$n\ge1$ a curve~$E_n'$
with equation $R_n^2= (T_n^2+a_{n-1}^2)/(T_n^2+b_{n-1}^2)$, isomorphic to
$E_n$ via~$\theta_n$ (defined as for $\theta=\theta_1$ as above);
these fit into a commutative diagram
$$
\xymatrix{
\cdots \ar[r] & E_n'\ar[r]^{\varphi_n'}\ar[d]^{\theta_n} & E_{n-1}'\ar[r]\ar[d]^{\theta_{n-1}} & \cdots \ar[r] & E_1' \ar[d]^{\theta_{1}}\ar[dr]^{\alpha} \\ 
\cdots \ar[r] & E_n\ar[r]^{\varphi_n} & E_{n-1}\ar[r] & \cdots \ar[r]
& E_1 \ar[r]^{\varphi_1} & E_0
}
$$ where $\varphi_n':E_{n}'\to E_{n-1}'$ is the $2$-isogeny which
makes the diagram commute. A little algebra shows that~$\varphi_n'$ is
given by
\[
   r_{n-1} = \frac{t_n^2+a_{n-1}a_{n-2}}{t_n^2+a_{n-1}b_{n-2}} 
          = \frac{a_{n-2}r_n^2-a_{n-1}}{-b_{n-2}r_n^2+a_{n-1}}, \qquad
   t_{n-1} = \frac{t_n}{r_n}.
\]

For any point sequence~$(P_n)$ (with $P_n\in E_n(\C)$ and
$\varphi_{n+1}(P_{n+1})=P_n$ for all~$n\ge0$) we set
$P_n'=(r_n,t_n)=\theta_n^{-1}(P_n)\in E_n'(\C)$ for $n\ge1$.  Since
$\alpha(P_1')=P_0$, we have
\[
   r_1^2 = \frac{x_0-e_3}{x_0-e_2}, \qquad \text{and}\qquad
   t_1 = -\frac{y_0}{2r_1(x_0-e_2)} = \sqrt{x_0-e_1};
\] 
note that these equations determine $r_1$ (and then~$t_1$) up to sign.
Next, from $\varphi_n'(P_n')=P_{n-1}'$ for $n\ge2$, we have
\[
   r_n^2 = \frac{a_{n-1}(r_{n-1}+1)}{b_{n-2}r_{n-1}+a_{n-2}}, 
           \qquad \text{and}\qquad   t_n   = r_{n}t_{n-1};
\]
again, these determine $(r_n,t_n)$ up to sign.

Hence we may construct all possible point sequences $(P_n')$ with
$P_n'\in E_n'(\C)$ for~$n\geq1$, starting from
$P_0=(x_0,y_0)\in E_0(\C)$ with $y_0\not=0$, by initialising
\[
   r_1 = \sqrt\frac{x_0-e_3}{x_0-e_2}, \qquad \text{and}\qquad
   t_1 = -\frac{y_0}{2r_1(x_0-e_2)}
\] 
to determine $P_1'=(r_1,t_1)$, and then iterating the following to
obtain $P_n'=(r_n,t_n)$ for~$n\ge2$:
\[
   r_n = \sqrt\frac{a_{n-1}(r_{n-1}+1)}{b_{n-2}r_{n-1}+a_{n-2}}, 
         \qquad \text{and}\qquad   t_n   = r_{n}t_{n-1}.
\]
Suitable choices of signs of~$r_n$ will be discussed below, which will
ensure that these sequences converge.  Then we will have
$r_{\infty}=\lim r_n=1$ and $t_{\infty}=\lim t_n$ satisfying
\[
   x_{\infty} =  t_{\infty}^2+\frac{2}{3}M^2, \qquad 
   y_{\infty} = -2t_{\infty}(t_{\infty}^2+M^2),   
\]
where $M=\AGM(a,b)$ as usual.  It follows that
\[
   t_{\infty} = \frac{-y_{\infty}/2}{x_{\infty}+M^2/3},
\]
as in the statement of Proposition~\ref{prop:elog-formula}.

\subsection{Choice of signs in the iteration}

We now show that we do obtain coherent, convergent sequences, provided
that for all (or all but finitely many)~$n$ we choose the sign
of~$r_n$ so that~$\Re(r_n)\ge0$; always assuming that the isogeny
sequence itself is optimal.

\begin{proposition}\label{prop:r-signs}
With the notation of the previous section, assume that the AGM
sequence satisfies $\Re(a_n/b_n)>0$ for all~$n\ge0$.

If $\Re r_n\ge0$ for all~$n\ge1$, then the point sequence
$(P_n)=(\theta_n(r_n,t_n))$ determined by the iteratively defined
sequence of pairs $(r_n,t_n)$ is coherent.

The same conclusion holds if $\Re r_n\ge0$ for all but finitely
many~$n\ge1$.
\end{proposition}
\begin{proof}
Recall that $\Lambda_n$ is the period lattice of~$E_n$ for $n\ge0$,
with $\Z$-basis $w_1,w_2$ such that $w_1=\pi/M(a_0,b_0)$ generates
$\cap_n\Lambda_n$, and $\Lambda_n=\ideal{w_1,2^nw_2}$ for all~$n\ge0$.
So for each~$n$ there exists $z_n\in\C$, uniquely determined
modulo~$\Lambda_n$, such that $x_n=\wp_{\Lambda_n}(z_n)$
and~$y_n=\wp_{\Lambda_n}'(z_n)$.  We wish to show that the $z_n$ may
be chosen independently of~$n$.

Since
\[
     r_n =   \frac{12x_n+5a_{n-1}^2-b_{n-1}^2}{12x_n+5b_{n-1}^2-a_{n-1}^2},
\]
we may regard~$r_n$ as the value at~$z_n$ of an elliptic
function~$f_n$ of degree~$2$ with respect to~$\Lambda_n$.  Similarly
its square,
\[
     r_n^2 = \frac{x_{n-1}-e_3^{(n-1)}}{x_{n-1}-e_2^{(n-1)}},
\]
is the value at~$z_n$ of $f_n^2$, which is an elliptic function with
respect to the larger lattice~$\Lambda_{n-1}$.  It follows that
\[
   f_{n}(z+2^{n-1}w_2)=-f_{n}(z)
\]
for all~$z\in\C$ and all~$n\ge1$.

Since 
\[
  \wp_n(0) = O_E = \theta_n((\infty,1))
\]
and 
\[
  \wp_n(w_1/2) = (e_1^{(n)},0) = \theta_n((0,a_{n-1}/b_{n-1})),
\]
we have $f_{n}(0)=1$ and $f_{n}(w_1/2)=a_{n-1}/b_{n-1}$ for
all~$n\ge1$.

We now consider the preimage~$\RR_n$ of the right half-plane
under~$f_{n}$, for $n\ge1$.  Since $f_{n}(w_1/2)=a_{n-1}/b_{n-1}$
and~$\Re(a_{n-1}/b_{n-1})>0$, this contains~$w_1/2$ for all~$n$.  Let
$\RR_n^o$ denote the connected component of~$\RR_n$ which
contains~$w_1/2$.  Both $\RR_n$ and~$\RR_n^o$ are invariant under
translation by~$w_1$ (by periodicity of~$f_{n}$), and $\RR_n$ is the
union of all translates of~$\RR_n^o$ by multiples of~$2^{n}w_2$.  The
preimage of the left half-plane under~$f_n$ is
$\LL_n=\RR_n+2^{n-1}w_2$, which is the union of the translates
of~$\RR_n^o$ by odd multiples of~$2^{n-1}w_2$.

Consider a point $P_n=\wp_n(z_n)\in E_n(\C)$, where $z_n\in\RR_{n}^o$.
Its preimages in~$E_{n+1}(\C)$ are $\wp_{n+1}(z_n)$ and
$\wp_{n+1}(z_{n}')$, where $z_{n}'=z_n+2^{n}w_2$.  One of $z_n$,
$z_{n}'$ lies in $\RR_{n+1}$, the other in $\LL_{n+1}$.  Since
$w_1/2\in\RR_k^o$ for all~$k$, one can show that
$\RR_{n}^o\subset\RR_{n+1}^o$ (see Lemma~\ref{lem:halfplane} below).
Hence, in fact, $z_n\in\RR_{n+1}^o$ and $z_{n}'\in\LL_{n+1}$.

Hence, by choosing the sign of each~$r_n$ for~$n\ge1$ so that it lies
in the right half-plane (for all~$n\ge1$), we ensure that each
$P_{n}=\wp_n(z_n)$, where $z_n\in\RR_1^o$ does not depend on~$n$.
Hence the associated point sequence is coherent, as required.

For the last part, if $\Re r_n>0$ only for~$n>n_0\ge0$, then we simply
apply the above argument to $E_{n_0}$ and $(P_n)_{n\ge n_0}$, noting
that $P_{n_0}$ is a lift of~$P_0$ to~$E_{n_0}(\C)$, and that every
elliptic logarithm of~$P_{n_0}$ is also one of~$P_0$.
\end{proof}

\begin{lemma}\label{lem:halfplane}
In the notation of Proposition~\ref{prop:r-signs},
$\RR_{n}^o\subset\RR_{n+1}^o$ for all $n\ge1$.
\end{lemma}
\begin{proof}
It suffices to show that $\Re f_{n+1}(z)$ has constant sign for
$z\in\RR_n^o$, since this sign is positive for $z=w_1/2\in\RR_n^o$.
If not, then there exists~$z\in\RR_n^o$ such that $\Re f_{n+1}(z)<0$,
so $f_{n+1}(z)^2$ is real and negative.  We show this to be
impossible.

We have
\[
     r_{n-1} =
     \frac{a_{n-2}r_{n}^2-a_{n-1}}{-b_{n-2}r_{n}^2+a_{n-1}} =
     h_n(r_{n-1}^2) = g_n(r_{n-1}), 
\]
say, where $h_n$ is the linear fractional transformation
\[
            z \mapsto \frac{a_{n-2}z-a_{n-1}}{-b_{n-2}z+a_{n-1}},
\]
and $g_n(z)=h_n(z^2)$.  This implies that
\[
       f_{n}(z) = g_{n+1}(f_{n+1}(z)) = h_{n+1}(f_{n+1}(z)^2).
\]

To complete the proof we show that the image of the negative real axis
under~$h_n$ is contained in the left half-plane, for all~$n\ge1$.  Let
$t\in\R$ be negative, and set $s=2t-1<-1$, and
$\alpha=a_{n-2}/b_{n-2}$; then
\[
   h_n(t) = \frac{s\alpha-1}{\alpha-s},
\]
and we leave it to the reader to check that this has negative
real part when $s<-1$ and $\Re\alpha>0$.
\end{proof}

We remark that this lemma implies that we always have $\Re r_n>0$
for~$n\ge2$.  It is possible to have $\Re r_1=0$; this occurs if and
only if~$x_0$ lies on the open line segment between~$e_2$ and~$e_3$.

\subsection{The elliptic logarithm algorithm}

We summarise this section with the following algorithm.

\begin{algorithm}[Complex Elliptic Logarithm]\label{algo:elog}
Given an elliptic curve $E$ defined over $\C$ by the Weierstrass
equation $Y^2=4(X-e_1)(X-e_2)(X-e_3)$, and a non-$2$-torsion point
$P\in E(\C)$, compute an elliptic logarithm of $P$.

\bigskip
\textbf{Input:} $E$, with roots $e_1,e_2,e_3$, and $P=(x_0,y_0)\in
E(\C)$, with $y_0\not=0$.
\begin{enumerate}
\item Set $a_0 = \sqrt{e_1-e_3}$ and $b_0 = \sqrt{e_1-e_2}$, choosing
  the numbering of the roots (if necessary) and the signs so that
  $|a_0-b_0|<|a_0+b_0|$.
\item Set $r = \sqrt{(x_0-e_3)/(x_0-e_2)}$, with $\Re r\ge0$.
\item Set $t = -y_0/(2r(x_0-e_2))$ (so $t^2=x_0-e_1$).
\item Repeat the following, for $n=1,2,\dots$:
\begin{enumerate}
\item set
\[
   a_n = \frac{a_{n-1}+b_{n-1}}{2},\quad b_n = \sqrt{a_{n-1}b_{n-1}},
\]
choosing the sign of $b_n$ so that $|a_n-b_n| < |a_n+b_n|$;
\item set $r \leftarrow \sqrt{{a_{n}(r+1)}/{(b_{n-1}r+a_{n-1})}}$,
  with $\Re r >0$.
\item set $t \leftarrow rt$.
\end{enumerate}
until $|a_n/b_n-1|$ and~$|r-1|$ are sufficiently small.  Set $M=\lim
a_n$.
\end{enumerate}

\textbf{Output:}
$$
   z_P = \frac{1}{M}\arctan\left(\frac{M}{t}\right).
$$
\end{algorithm}

Note that the output value of $z_P$ may not be in the fundamental
parallelogram of the period lattice $\Lambda$.  However, assuming that
the usual range for the $\arctan$ function is used, where $-\pi/2 <
\Re\arctan(x)\le\pi/2$, we will have $z_P=xw_1+iyw_1$ with $x,y\in\R$
and~$-1/2<x\le1/2$.

For points~$P$ of order~$2$, choose the labelling of the roots so that
$P=(e_1,0)$ and then take $z_P=w_1/2=\pi/(2M)$ where $M=
M(\sqrt{e_1-e_3}, \sqrt{e_1-e_2})$.

\subsection{The real case}

For elliptic curves defined over~$\R$ there is some advantage in
adapting the algorithm to use real arithmetic where possible, even
though the algorithm as given above works perfectly well in this
situation.  We divide into cases as in sections \ref{sec:special-i}
and \ref{sec:special-ii} above.

\subsubsection{Curves with positive discriminant}
Order the roots, which are all real, as in
section~\ref{sec:special-i}, so that $e_1>e_2>e_3$; the real and
imaginary periods~$w_1,w_2$ are then given
by~\eqref{eqn:pos-real-periods}.

Let $P=(x_0,y_0)\in E(\R)$ with $2P\not=0$ (so $y_0\not=0$).  If $P$
is in the connected component of the identity of $E(\R)$ then
$x_0>e_1$, and it is immediate from the formulae given above that as
well as all $a_n,b_n$ being real and positive, so too are all $r_n$,
and the $t_n$ are real and with constant sign (opposite to that
of~$y_0$).  Hence~$z_P$, the output of the algorithm, is real and in
the interval $|z_P|<w_1/2$.

Now suppose that $e_2>x_0>e_3$, so that $P$ is in the other real
component.  Now $z_P=x_P+w_2/2$ where $x_P$ is real, and it suffices
to compute~$x_P$.  To do this we may replace~$P$ by~$P'=P+(e_3,0)$
which is in the identity component and has elliptic logarithm equal
to~$x_P$.  A short calculation shows that we may compute~$x_P$ using
the usual iteration, with the positive real initial values
\[
     r' = a_0/\sqrt{e_1-x_0}; \qquad t'=r'y_0/2(x_0-e_3).
\]

\subsubsection{Curves with negative discriminant}
As in \ref{sec:special-ii}, we order the roots so that $e_1$ is real
and $\Im e_2>0$.  Set $a_0=\sqrt{e_1-e_3}=x+yi$ where $x,y>0$.  The
real period is $w_1=\pi/M(a_0,b_0)=\pi/M(x,R)$ where $R=|a_0|$.  Now
let $\sqrt{x_0-e_3}=u+iv$ with $u,v>0$, and then set the initial
values of $r$ and~$t$ to $r_1=(u+iv)/(u-iv)$ and
$t_1=-y_0/2(u^2+v^2)$.

Applying the first step in the iteration, we find that $a_1=x$ and
$b_1=R$, and also that $r_2=\sqrt{ux/(ux+vy)}$, where the quantity
inside the square root is real and positive, so we may take $r_2>0$
also, and $t_2=r_2t_1$ which is also real and with the opposite sign
to~$y_0$.  Now the rest of the iteration may be carried out using real
values for all quantities, and again the output value~$z_P$ is real
and satisfies $|z_P|<w_1/2$.

\section{Examples}
In the following examples, we will illustrate our method for computing
the period lattices of elliptic curves over $\C$, and the elliptic
logarithms of complex points. These examples were computed both using
the \Magma\ implementation by the second author and also using the
\Sage\ implementation by the first author;  a \Sage\ script
which reproduces these examples is available at \cite{elog-arxiv-preprint}.

All complex numbers in our examples were first computed to $100$
decimal places (though we only show the first $20$ decimal places
below, to save space) and then to $200$, $400$ places and up to~$1600$
decimal places: this allows us to illustrate the rapid convergence in
practice, where each iteration doubles the number of correct decimal
places.  In all these examples, no more than~$11$ iterations were required to
obtain $1600$ decimal places, making the computations essentially
instantaneous in practice.  Note that in all our examples the input
consists of exact algebraic values of $e_1$, $e_2$ (and
$e_3=-e_1-e_2$) in either $\Z[i]$ or $\Z[\root3\of2]$, which allows us
to regard the input as having infinite precision.  We have not
determined how sensitive the algorithms are to imprecision in the
input, since for our main application (elliptic curves over number
fields) the input is exact in the above sense, but we note that in
\cite[Prop.~3.3]{Dupont-thesis} one may find a formula for the number
of iterations required to approximate $\AGM(a,b)$ to a given relative
bit-precision, in terms of $\log(a/b)$.

Note also that we had to implement functions for computing optimal AGM
values, as the standard AGM function in \Magma\ does not always return
an optimal one, and this was also true of \Sage\ until version 4.3.2
when our new implementation (jointly written with Robert Bradshaw) was
released.

\begin{example}
Let $E$ be the elliptic curve over $\C$ given by the Weierstrass equation
$$
E:\quad Y^2 = 4(X-e_1)(X-e_2)(X-e_3)
$$
with
$$
e_1 = 3-2i,\quad e_2 = 1+i,\quad e_3 = -4+i.
$$
Observe that $\sum_{j=1}^3 e_j=0$. We will compute the period lattice
of $E$ using the method described in Theorem~\ref{thm:w1w2w3}. To do this,
first we let $E_0=E$ and calculate
$$
a_0 = \sqrt{e_1-e_3},\quad b_0 = \sqrt{e_1-e_2},\quad c_0 = \sqrt{a_0^2-b_0^2},
$$
where the signs of $a_0,b_0,c_0$ are chosen so that
\eqref{eqn:abc-cond} holds:
$$
|a_0-b_0|\le|a_0+b_0|,\quad |a_0-c_0|\le|a_0+c_0|,\quad
|c_0-ib_0|\le|c_0+ib_0|.
$$
In this example, one can verify that such $a_0,b_0,c_0$ are
\begin{eqnarray*}
a_0 &=& 2.70331029534753078867\ldots - i0.55487525889334275023\ldots\\
b_0 &=& 1.67414922803554004044\ldots - i0.89597747612983812471\ldots\\
c_0 &=& 2.23606797749978969640\ldots.
\end{eqnarray*}
In fact, all conditions in \eqref{eqn:abc-cond} are strictly
inequalities in this case, as the period lattice of $E$ is
non-rectangular.  Using Theorem \ref{thm:w1w2w3} with optimal AGM
values, we compute
\begin{eqnarray*}
w_1 &=& 1.29215151748713051904\ldots + i0.44759218107818896608\ldots\\
w_2 &=& 1.42661373451784507587\ldots - i0.80963848056301882107\ldots\\
w_3 &=& -0.13446221703071455682\ldots + i1.25723066164120778715\ldots;
\end{eqnarray*}
any two of $w_j$ form a $\Z$-basis for $\Lambda$ (the period lattice
of $E$), and, as expected, these $w_j$ are minimal coset
representatives of $2\Lambda$ in $\Lambda$.

Computing each $w_j$ to $100$ (respectively $200$, $400$, $800$,
$1600$) decimal places requires only~$7$ (respectively $8$, $9$, $10$,
$11$) basic AGM iterations.  We verified that the first $100$
(respectively $200$, $400$, $800$) decimal places are unchanged when
recomputed to higher precision, and also that the equality
$w_1=w_2+w_3$ held to the required number of decimal places in each
case.

\medskip
Next, we compute an elliptic logarithm of the point
$$
P = (2-i, 8+4i)\in E(\C)
$$ (which has infinite order). Using $a_0,b_0$ as above,
Algorithm~\ref{algo:elog} gives
$$
z_P = -0.72212997914002299126\ldots + i0.01717122412650902249\ldots.
$$ The number of iterations required for $100$, \dots, $1600$ decimal
places is the same as for the AGM itself, namely $7$,\dots,$11$.  

Note that $z_P$ is only well-defined modulo $\Lambda$.  Depending on
the basis for $\Lambda$, the value~$z_P$ obtained using
Algorithm~\ref{algo:elog} may not lie in the fundamental parallelogram
spanned by that basis.  In our case, one can check that
\begin{eqnarray*}
z_P &=& (-0.33249952362000772434\ldots)w_1 - (0.20502411273191295799\ldots)w_2\\
&\equiv& (0.66750047637999227565\ldots)w_1 + (0.79497588726808704200\ldots)w_2,
\end{eqnarray*}
and so $z_P$ is not in the fundamental parallelogram spanned by
$\{w_1,w_2\}$. Finally, one may verify that, to the given precision,
we have, as expected, 
\[
   \wp_\Lambda(z_P)=x(P),\quad 
   \wp'_\Lambda(z_P)=y(P),
\]
and also
\[
  \wp_\Lambda(w_1/2)=e_1, \quad
  \wp_\Lambda(w_2/2)=e_2, \quad
  \wp_\Lambda(w_3/2)=e_3,
\]
with $\wp'_\Lambda(w_j/2)= 0$ for all $j=1,2,3$.
\end{example}

\begin{example}[Rectangular Lattice]
Let $E$ be the elliptic curve over $\C$ given by the Weierstrass equation
$$
E:\quad Y^2 = 4(X-e_1)(X-e_2)(X-e_3)
$$
with
$$
e_1 = 1+3i,\quad e_2 = -4-12i,\quad e_3 = 3+9i.
$$
Observe that $\sum_{j=1}^3 e_j=0$ and the $e_j$ are collinear. By letting
$E_0=E$ and computing $a_0,b_0,c_0$ as before, we have
\begin{eqnarray*}
a_0 &=& 1.47046851723128684330\ldots - i2.04016608641756892919\ldots\\
b_0 &=& -3.22578581905571472955\ldots - i2.32501487101070997214\ldots\\
c_0 &=& 2.75099469475848456460\ldots - i3.81680125374499001591\ldots.
\end{eqnarray*}
This time, however, we have $|a_0-b_0|=|a_0+b_0|$, while the other two
relations in \eqref{eqn:abc-cond} are strict inequalities. Hence we have
two minimal elements (up to sign) in one coset of $2\Lambda$ in $\Lambda$ (where
$\Lambda$ is the period lattice of $E$), and $\Lambda$ is rectangular.

To obtain an orthogonal basis for $\Lambda$, we first let $w,w'=\pi/M(a_0,\pm
b_0)$:
\begin{eqnarray*}
w &=& -0.29920293143872535713\ldots + i1.10940038117892953702\ldots\\
w' &=& 1.14708588706988127437\ldots + i0.06697438037476960963\ldots.
\end{eqnarray*}
One can check that $|w|=|w'|$. Let $w_1 = (w+w')/2$ and $w_2 =
(w-w')/2$. Then $w_1,w_2$ form an orthogonal basis for $\Lambda$, as
in Lemma~\ref{lem:ortho-basis}:
\begin{eqnarray*}
w_1 &=&  0.42394147781557795862\ldots + i0.58818738077684957333\ldots\\
w_2 &=& -0.72314440925430331575\ldots + i0.52121300040207996369\ldots.
\end{eqnarray*}
Note that $\Re(w_2/w_1)=0$, as required for orthogonality.

\medskip
Let $z_P$ be an elliptic logarithm of the point $P=(3+2i, 28-14i)\in
E(\C)$ (which $P$ has infinite order). Algorithm \ref{algo:elog} gives
\begin{eqnarray*}
z_{P} &=& -0.42599662534207481578\ldots - i0.02491254923738153924\ldots\\
 &\equiv&  (0.62858224538977667533\ldots)w_1 + (0.37134662195976180031\ldots)w_2.
\end{eqnarray*}
Finally, we verify that (within the working precision)
$$
\wp_\Lambda(z_P)\approx x(P),\quad
\wp'_\Lambda(z_P)\approx y(P),
$$
and also
\begin{eqnarray*}
\wp_\Lambda\left(w_1/2\right) \approx e_2,\\
\wp_\Lambda\left(w_2/2\right) \approx e_3,\\
\wp_\Lambda\left(w/2\right) \approx e_1,
\end{eqnarray*}
and 
$$
\wp'_\Lambda(w_1/2)\approx
\wp'_\Lambda(w_2/2)\approx
\wp'_\Lambda(w/2)\approx 0.
$$
\end{example}

\begin{example}
Let $K=\Q(\theta)$ where $\theta$ is a root of the polynomial $x^3-2$. Let $E$
be the elliptic curve defined over $K$ given by the Weierstrass equation
$$
E:\quad Y^2 = 4(X-\theta)(X-1)(X+1+\theta).
$$
Note that $K$ has one real embedding and one
pair of complex embeddings. Let $E_1, E_2$ be the real and complex embedding of
$E$ respectively, with equations
\begin{eqnarray*}
E_1: & & Y^2 = 4(X-\sqrt[3]{2})(X-1)(X+1+\sqrt[3]{2})\\
E_2: & & Y^2 = 4(X-\omega\sqrt[3]{2})(X-1)(X+1+\omega\sqrt[3]{2})
\end{eqnarray*}
where $\sqrt[3]{2}$ is the real cube root of~$2$ and $\omega=\exp(2\pi
i/3)$ is a cube root of unity. Now $E_1$ has three real roots, so
the period lattice of $E_1$ is rectangular. In fact, by letting
$e^{(0)}_1=\sqrt[3]{2}, e^{(0)}_2=1, e^{(0)}_3=-1-\sqrt[3]{2}$, we can
compute $a_0,b_0,c_0$ satisfying \eqref{eqn:abc-cond} as
\begin{eqnarray*}
a_0 &=& 1.87612422291002530767\ldots\\
b_0 &=& 0.50982452853395859808\ldots\\
c_0 &=& 1.80552514518487755254\ldots.
\end{eqnarray*}
One sees that $|c_0-ib_0|=|c_0+ib_0|$. As before, we compute
\begin{gather*}
w = \frac{\pi}{M(c_0,ib_0)} =
 2.90130425944817643666\ldots - i1.70677932803214980295\ldots\\
w' = \frac{\pi}{M(c_0,-ib_0)} = \bar{w},
\end{gather*}
and let $w_1,w_2 = (w\pm w')/2$. Then $w_1,w_2$ form an orthogonal
basis for the period lattice of $E_1$. In this example, we have
$w_1=\Re(w)$ and $w_2=i\Im(w)$.

\medskip
Secondly, the period lattice of $E_2$ is non-rectangular, since
the roots of $E_2$ are not collinear. In fact, by letting
$e^{(0)}_1=-1-\omega\sqrt[3]{2}, e^{(0)}_2=1,
e^{(0)}_3=\omega\sqrt[3]{2}$ (here we must ensure that $a_0,b_0,c_0$
satisfy \eqref{eqn:abc-cond}), we have
\begin{eqnarray*}
a_0 &=& 1.10851094368231305521\ldots - i0.98431471713501219051\ldots\\
b_0 &=& 0.43669517024285334726\ldots - i1.24929666083200513980\ldots\\
c_0 &=& 1.34004098848655674756\ldots - i0.40712323180652750769\ldots.
\end{eqnarray*}
One can check that all conditions in \eqref{eqn:abc-cond} are strict
inequalities, which also confirms that the period lattice of $E_2$ is
non-rectangular. By Theorem \ref{thm:w1w2w3}, we finally obtain
\begin{eqnarray*}
w_1 &=& 1.28194824894788708942\ldots + i1.88277404359595361782\ldots\\
w_2 &=& 2.36557653380849535471\ldots - i0.03808700290170419307\ldots\\
w_3 &=& -1.08362828486060826529\ldots + i1.92086104649765781090\ldots
\end{eqnarray*}
with $w_1\approx w_2+w_3$.
\end{example}

\begin{example}\label{eq:special-ii}
Let $E$ be the elliptic curve over $\C$ given by the Weierstrass equation
$$
E:\quad Y^2 = 4(X-e_1)(X-e_2)(X-e_3)
$$
with
$$
e_1 = -1-3i,\quad e_2 = 3+i,\quad e_3 = -2+2i.
$$ Observe that $\sum_{j=1}^3 e_j=0$ and $|e_1-e_3|=|e_2-e_3|$. Thus
$e_1,e_2,e_3$ form an isosceles triangle. Letting $E_0=E$ and
computing $a_0,b_0,c_0$ as before, we have
\begin{eqnarray*}
a_0 &=& 1.74628455779589152702\ldots - i1.43161089573822132705\ldots\\
b_0 &=& 0.91017972112445468260\ldots - i2.19736822693561993207\ldots\\
c_0 &=& 2.24711142509587014360\ldots - i0.22250788030178260411\ldots.
\end{eqnarray*}
Hence by Theorem \ref{thm:w1w2w3}, we obtain
\begin{eqnarray*}
w_1 &=&  0.81646689790312614904\ldots + i1.10773333340066743861\ldots\\
w_2 &=&  1.36061503191563570645\ldots - i0.20595647167234558716\ldots\\
w_3 &=& -0.54414813401250955741\ldots + i1.31368980507301302578\ldots
\end{eqnarray*}
with $w_1\approx w_2+w_3$. In addition, one can check that
$\Re(w_1/w_3)=1/2$ as claimed in Section \ref{sec:special-ii}. Let $\Lambda$ be the period lattice of $E$. We finally verify that
$\wp_\Lambda(w_j/2) \approx e_j$ for $j=1,2,3$, and
$$
\wp'_\Lambda(w_1/2)\approx
\wp'_\Lambda(w_2/2)\approx
\wp'_\Lambda(w_3/2)\approx0.
$$
\end{example}

\providecommand{\bysame}{\leavevmode\hbox to3em{\hrulefill}\thinspace}
\providecommand{\MR}{\relax\ifhmode\unskip\space\fi MR }
\providecommand{\MRhref}[2]{%
  \href{http://www.ams.org/mathscinet-getitem?mr=#1}{#2}
}
\providecommand{\href}[2]{#2}

\end{document}